\documentclass{amsart}

\usepackage{tikz}
\usetikzlibrary{graphs,matrix}
\usepgflibrary{arrows}

\usepackage{amsthm}
\usepackage{amsfonts}
\usepackage{amsmath,amssymb}
\usepackage{indentfirst,latexsym,bm,amsthm,graphicx,colortbl}
\usepackage{fancyhdr}
\usepackage{times}
\usepackage{amsmath,amsfonts,amssymb,amsthm}
\usepackage{epsfig}
\usepackage[all]{xy}
\newtheorem{theorem}{Theorem}[section]
\newtheorem{lemm}[theorem]{Lemma}

\newtheorem{theo}[theorem]{Theorem}
\theoremstyle{definition}

\newtheorem{coro}[theorem]{Corollary}
\theoremstyle{remark}
\newtheorem{remark}[theorem]{Remark}

\numberwithin{equation}{section}

\usepackage{leftidx}
\usepackage{amscd}
\usepackage{multicol}
\usepackage{amsbsy}
\usepackage{epsfig}
\usepackage{tikz}
\usepackage{cite}
\usepackage{indentfirst}
\usepackage{multirow}
\usepackage{tabularx}
\usepackage{mathrsfs}
\usepackage{amsfonts}
\usepackage{}
\usepackage{amsmath}
\usepackage{wasysym,amssymb,eufrak,indentfirst,graphicx,cite,amsthm,color}
\usepackage{array}
\begin{document}

\baselineskip 17pt

\title[Poincar\'{e} series, exponents, and McKay-Slodowy correspondence]
{Poincar\'{e} series, exponents of affine Lie algebras, and McKay-Slodowy correspondence}

\author[Jing]{Naihuan Jing}
\address{Department of Mathematics,
   North Carolina State University,
   Ra\-leigh, NC 27695-8205, USA}
\email{jing@math.ncsu.edu}

\author[Wang]{Danxia Wang}
\address{Department of Mathematics, Shanghai University,
Shanghai 200444, China}
\email{wangdanxia@i.shu.edu.cn}

\author[Zhang]{Honglian Zhang}
\address{Department of Mathematics, Shanghai University,
Shanghai 200444, China} \email{hlzhangmath@shu.edu.cn}

\subjclass[2010]{14E16, 17B67, 20C05}

\keywords{Poincar\'{e} series, tensor algebras, invariants, McKay-Slodowy correspondence.}
\begin{abstract}
Let $N$ be a normal subgroup of a finite group $G$ and
$V$ be a fixed finite-dimensional
$G$-module. 
The Poincar\'{e} series for the multiplicities of induced modules and restriction modules in the tensor algebra $T(V)=\oplus_{k \geq 0}V^{\otimes k}$
are studied in connection with the McKay-Slodowy correspondence. In particular, it is shown that
the closed formulas for the Poincar\'e series associated with the distinguished pairs of subgroups
of $\mathrm{SU}_2$ give rise to the exponents of all untwisted and twisted affine Lie algebras
except ${\rm A}_{2n}^{(1)}$.

\end{abstract}

\date{}
\maketitle

\section{Introduction}

McKay correspondence establishes a far-reaching one-to-one map between subgroups of the special unitary group
$\mathrm{SU}_2$ and affine Dynkin diagrams of simply laced types \cite{Mc}. Since its introduction, numerous
deep connections and applications have been found in combinatorics, algebraic geometry, representation theory and
mathematical physics. For example, McKay's observation corresponds to the classification of the minimal resolutions of the singularity of
the action of a finite subgroup $G$ of $\mathrm{SU}_2$ on $\mathbb C^2$ in terms of simply laced Dynkin diagrams
\cite{G-SV,Kn,Kos2,Sl2,Ste}.

Let $G$ be a finite group and $V$ a faithful $G$-module. The tensor algebra $T(V)=\oplus_{k\geq 0}V^{\otimes k}$ is naturally
a $G$-module. Similar to the well-known result of Molien \cite{B} on $G$-invariants of
the symmetric algebra $S(V)$, the Poincar\'e series $m_{V}(t)$ of $G$-invariants in the tensor algebra $T(V)$ are rational functions in terms of
irreducible characters of $G$. In particular, when $G$ is a finite subgroup of $\mathrm{SU}_2$, the Poincar\'e series  $m_{V}(t)$ provides
a conceptual interpretation of the exponents of the affine Lie algebra in simply laced types \cite{Ben}.
It is natural to expect similar explanation of the exponents for other types of affine Lie algebras, in particular twisted
affine Lie algebras. 

In the classic work \cite{Sl} Slodowy studied the minimal resolution of the singularity of the action of the quotient $G/N$ of finite subgroups
of $\mathrm{SU}_2$ on the quotient space
$\mathbb C^2/N$, which generalizes the minimal resolution of the singularity of $\mathbb C^2/G$, where $G$ is a finite subgroup of $\mathrm{SU}_2$.
Slodowy discovered that the minimal resolution of singularity of $\mathbb C^2/N$ by
the action of $G/N$ are in one-to-one correspondence to all affine Dynkin diagrams.
Algebraically, Slodowy found that the pairs of finite groups $N, G$ of $\mathrm{SU}_2$, where $N$ is a normal subgroup of $G$, essentially realize all affine Dynkin diagrams.

To be more specific, let $G$ be a finite group with a normal subgroup $N$.  The Grothendieck group
$R(G)$ is spanned by the irreducible complex modules $\rho_i$, $i\in {\rm I}_{G}$.
Similarly $R(N)$ is spanned by the irreducible complex modules $\phi_i$, $i\in {\rm I}_{N}$.
The induction functor $\mathrm{Ind}: R(N)\longrightarrow R(G)$ and the restriction functor $\mathrm{Res}: R(G)\longrightarrow R(N)$
are defined as usual and we will denote the images $\mathrm{Ind}(\phi):=\hat{\phi}$ and $\mathrm{Res}(\rho):=\check{\rho}$ respectively.

It turns out that the dimension of the span $\{\check{\rho}_i|i\in {\rm I}_G\}$ agrees with that of the span $\{\hat{\phi}_i|i\in {\rm I}_N\}$. We pare down the
set $\{\check{\rho}_i|i\in {\rm I}_G\}$ to a basis $\{\check{\rho}_i|i\in {\rm \check{I}}\}$ for the subspace $\mathrm{span}(\check{\rho}_i)$,
and similarly we also pare down a corresponding basis $\{\hat{\phi}_i|i\in {\rm \hat{I}}\}$ for the subspace $\mathrm{span}(\hat{\phi}_i)$, then $|{\rm \hat{I}}|=|{\rm \check{I}}|$ is equal to the
common dimension of the two spans.

Let $V$ be a fixed finite-dimensional $G$-module, and we also denote $\mathrm{Res} V=\check{V}\in R(N)$.
Clearly $\check{V}\otimes \mathrm{span}(\check{\rho_i})\subset \mathrm{span}(\check{\rho_i})$ and
$V\otimes \mathrm{span}(\hat{\phi_i})\subset \mathrm{span}(\hat{\phi_i})$.
Therefore the tensor decompositions
\begin{equation}\label{1.2}
  \check{V}\otimes \check\rho_j=\bigoplus\limits_{i\in\check{\rm I}}a_{ij} \check\rho_{i}\qquad {\rm and}\qquad
  V\otimes \hat\phi_j=\bigoplus\limits_{i\in\hat{\rm I}}b_{ij} \hat\phi_i
\end{equation}
give rise to two integral matrices ${\rm \widetilde{A}}=(a_{ij})$ and ${\rm \widetilde{B}}=(b_{ij})$ of the same size respectively.
The corresponding representation graph is the digraph
$\mathcal{R}_{\check{V}}(\check{G})$ (resp. $\mathcal{R}_{V}(\hat{N})$) with vertices indexed by ${\rm \check{I}}$ (resp. ${\rm \hat{I}}$),
where $i$ is joined to $j$ by $\mathrm{max}(a_{ij}, a_{ji})$ (resp. $\mathrm{max}(b_{ij}, b_{ji})$) edges with an arrow pointing to
$i$ if $a_{ij}>1$ (resp. $b_{ij}>1$).

When the pair of subgroups are subgroups of $\mathrm{SU}_2$, Slodowy \cite{Sl} observed that the representation graphs
$\mathcal{R}_{\check{V}}(\check{G})$ and $\mathcal{R}_{V}(\hat{N})$ are in fact the affine Dynkin diagrams
(his original statement missed a couple of groups, cf. \cite{EtFe}). Of course, when $N=1$, the graphs $\mathcal{R}_{V}({G})$
are of simply-laced types according to the McKay correspondence. For nontrivial $N$, the explicit
correspondence goes as follows. Let $C_n$ be the cyclic group of order $n$, $D_n$ the binary dihedral group of
order $4n$, $T$ the binary tetrahedron group of order $24$ and $O$ the binary octahedral group of order $48$.
For $(G, N)=(D_{2(n-1)},D_{n-1})$, $(D_n,C_{2n})$,
$(D_{2n},C_{2n})$, $(O,T)$, $(T,D_2$), $(D_2,C_2$),
the representation graph $\mathcal{R}_{\check{V}}(\check{G})$ is
the twisted affine Dynkin diagram of type
${\rm A_{2n-1}^{(2)}}$, ${\rm D_{n+1}^{(2)}}$, ${\rm A_{2n}^{(2)}}$, ${\rm E_6^{(2)}}$, ${\rm D_4^{(3)}}$,
${\rm A_2^{(2)}}$ respectively,
and the representation graph $\mathcal{R}_{V}(\hat{N})$ is
the non-twisted multiply laced affine Dynkin diagram of type
${\rm B_n^{(1)}}$, ${\rm C_n^{(1)}}$, ${\rm C_n^{(1)}}$, ${\rm F_4^{(1)}}$, ${\rm G_2^{(1)}}$,
${\rm A_1^{(1)}}$ respectively.
Moreover, ${\rm C_{\widetilde{A}}=2I}-{\rm \widetilde{A}}$ and
${\rm C_{\widetilde{B}}}={\rm 2I}-{\rm \widetilde{B}}$
are the corresponding Cartan matrices of Dynkin diagrams
$\mathcal{R}_{\check{V}}(\check{G})$ and $\mathcal{R}_{V}(\hat{N})$, respectively.

To understand Slodowy's idea and prepare for the later work,  we provide a detailed exposition of the McKay-Slodowy correspondence in the first part.

The second part of the paper aims to generalize Benkart's interpretation of the exponents 
to all affine Dynkin diagrams. Let $V$ be as above and 
for simplicity we will also write
$\check{V}=V$.
For each $j\in {\rm \check{I}}$ (resp. ${\rm \hat{I}}$), let $\check{m}^{j}_{k}$ (resp. $\hat{m}^{j}_{k}$) be the multiplicity
of $\check{\rho}_j$ inside $V^{\otimes k}$ (resp. $\hat{\phi}_j$ inside $V^{\otimes k}$):
\begin{align*}
\check{m}^{j}_{k}={\mathrm{dim}}(\mathrm{Hom}_{N}(\check{\rho}_j, {V}^{\otimes k})) \ \ ({\rm resp.} \ \  \hat{m}^{j}_{k}={\mathrm{dim}}(\mathrm{Hom}_{G}(\hat{\phi}_j, {V}^{\otimes k}))).
\end{align*}
Therefore the generating series
\begin{equation}\label{1.3}
  \check{m}^{j}(t)=\sum\limits_{k\geq 0}\check{m}^{j}_{k}t^{k}\qquad {\rm and}\qquad
  \hat{m}^{j}(t)=\sum\limits_{k\geq 0}\hat{m}^{j}_{k}t^{k}
\end{equation}
are the Poincar\'{e} series for the multiplicities of $\check\rho_j$ and $\hat\phi_j$ in the tensor algebra
$T(V)=\oplus_{k \geq 0}V^{\otimes k}$ respectively.

In the special case when $N=1$ and $G$ is a finite group of $\mathrm{SU}_2$, the Poincar\'e series $\check{m}^j(t)=\hat{m}^j(t)$ has been thoroughly studied by Benkart \cite{Ben} where she
gave a closed rational function formula of $m^j(t)$ in terms of the irreducible characters of $G$. Moreover, it turns out that the
denominator also gives the detailed information about the spectrum of the Coxeter element, just as the
Poincar\'e series of the $G$-invariants in the symmetric algebra $S(V)$ does \cite{Kos,Kos2,Kos3,Sp1,Sp2,Ste,St,Su}.
Our main result will study the Poincar\'e series $\check{m}^j(t)$ and  $\hat{m}^j(t)$ for all the distinguished
pairs of subgroups of $\mathrm{SU}_2$ and provides a conceptual interpretation of
the exponents for all non-simply laced and twisted affine Lie algebras.

This paper is organized as follows. In Section $2$,
after a brief review of McKay correspondence we provide a detained exposition for the McKay-Slodowy correspondence
to realize all affine Dynkin diagrams.
To prepare for the later work, we clarify some of the missing points in the literature
(e.g. our treatment of types ${\rm A_2^{(2)}}$ and ${\rm A_{2n}^{(2)}}$) and emphasize the group theoretical
construction using the induction and restriction functors as much as possible.
In Section $3$ the formulas of
the Poincar\'{e} series for the $N$-restriction of the irreducible $G$-modules and the induced modules of the irreducible $N$-modules
in tensor algebra $T(V)=\oplus_{k \geq 0}V^{\otimes k}$ are given using the general theory.
In Section $4$, the Poincar\'e series $\check{m}^j(t)$ and $\hat{m}^j(t)$ are explicitly computed for the tensor
algebra $T(V)$ for all the distinguished pairs of subgroups $N \lhd G$ in $\mathrm{SU}_{2}$, and finally we show that
the closed formulas of the Poincar\'e series of the invariants in the tensor algebra $T(V)$
for the distinguished pairs of subgroups of $\mathrm{SU}_2$ give rise to
the exponents of all affine Lie algebras in both untwisted and twisted types except $A_{2n}^{(1)}$.

\section{Realizations of non-simply laced affine Dynkin diagrams}

\subsection{Simply-laced types}

\qquad

The McKay correspondence establishes a fundamental relation between finite subgroups of $\mathrm{SU}_2(\mathbb C)$ and affine Dynkin diagrams
of simply-laced types.
We recall the algebraic version to prepare for further development.

\subsubsection{Cyclic groups}

For $n\geq 2$, let $C_n=\langle z| z^n=1\rangle$ be the cyclic group of order $n$. The canonical imbedding $\pi: C_n\longrightarrow
\mathrm{SU}_2$ is given by $\pi(z)=diag(\theta_n^{-1}, \theta_n)$,
where $\theta_n=e^{2\pi \sqrt{-1}/{n}}$, a primitive $n$th root of unity.
The cyclic group $C_n$ has exactly $n$ complex irreducible modules $\xi_i $ ($i=0,1,\ldots, n-1$), which are all one-dimensional
given by $\chi_{\xi_i} (z)=\theta_n^i$ ($i=0,1,\ldots, n-1$).
Then $\pi\simeq\xi_1\oplus\xi_{-1}$ and the multiplication rule $\xi_l\otimes\xi_k=\xi_{l+k}$
implies the fusion rule
\begin{equation*}
\pi\otimes\xi_i=\xi_{i-1}\oplus\xi_{i-1}
\end{equation*}
which gives rise to the Dynkin diagram of type ${\rm A}_{n-1}^{(1)}$.

\hskip6pt

\begin{center}
\begin{tikzpicture}
\matrix (m) [matrix of math nodes, row sep=2mm, column sep=2cm]
   { \xi_0& \xi_1  \\};
\path (m-1-1) edge[double] (m-1-2);
\draw [->,>=angle 90] (1,0)--(1.01,0);
\draw [->,>=angle 90] (-1,0)--(-1.01,0);
\draw node at(3.2,0){$(n=2)$};
\end{tikzpicture}
\end{center}

\begin{center}
\begin{tikzpicture}
\matrix (m) [matrix of math nodes, row sep=2mm, column sep=1cm]
   { & & \xi_0 & & \\
   \xi_1 & \xi_2 & \cdots & \xi_{n-2} & \xi_{n-1} \\};
\path
(m-2-1) edge (m-2-2) (m-2-2) edge (m-2-3) (m-2-3) edge (m-2-4) (m-2-4) edge (m-2-5)
(m-1-3) edge (m-2-1) (m-1-3) edge (m-2-5);
\draw node at(4.8,-0.3){$(n>2)$};
\end{tikzpicture}
\end{center}

\subsubsection{Binary dihedral groups}\label{s:binary}
The binary dihedral group $D_n$ ($n\geq 2$) of order $4n$ is the group $\langle x, y| x^n=y^2=-1, yxy^{-1}=x^{-1}\rangle$.
The embedding $\pi$ of $D_n$ into $\mathrm{SU}_2$ is given by
\begin{equation*}\label{ddgen}
  \pi(x)=\left(
      \begin{array}{cc}
        \theta_{2n}^{-1} & 0 \\
        0 & \theta_{2n} \\
      \end{array}
    \right), \ \ \ \
  \pi(y)=\left(
      \begin{array}{cc}
        0 &\sqrt{-1}\\
       \sqrt{-1}& 0 \\
      \end{array}
    \right).
\end{equation*}
Clearly $\{\pm1, x^{i}$ $(i=1,\ldots, n-1)$,
$y$, $yx\}$ is a set of representatives of
the conjugacy classes of $D_{n}$.

The $n+3$ irreducible modules are realized as follows. Let $\delta_i =\hat\xi_i$, the induced module of the $i$th irreducible module $\xi_i$ of
the cyclic group $C_{2n}=\langle x\rangle$, a normal subgroup of index $2$. The characters are determined by
$\chi_{\delta_i} (x)=\theta_{2n}^i+\theta_{2n}^{-i}$ and $\chi_{\delta_i} (y)=0$. Consequently $\delta_i =\delta_{2n-i}$
for $i=0, 1, \ldots, n$. It is easy to see that $\delta_i $ are 2-dimensional irreducible $D_n$-modules for $1\leq i\leq n-1$, but $\delta_0$ and $\delta_n$ are decomposed into a sum of two irreducible one-dimensional modules:
\begin{align*}
\delta_0&=\delta_{0}^{+}\oplus \delta_{0}^{-}, \\
\delta_n&=\delta_{n}^{+}\oplus \delta_{n}^{-}.
\end{align*}
Then the irreducible components are determined by
\begin{align*}
\chi_{\delta_{0}^{\epsilon}}(x)&=1, \qquad \chi_{\delta_{0}^{\epsilon}}(y)=\epsilon,  \\
\chi_{\delta_{n}^{\epsilon}}(x)&=-1, \qquad \chi_{\delta_{n}^{\epsilon}}(y)=\epsilon\sqrt{(-1)^n},
\end{align*}
where $\epsilon=\{\pm\}$ and we note that $\chi_{\delta_{n}^{\epsilon}}(y^2)=\chi_{\delta_{n}^{\epsilon}}(x^n)=(-1)^n$,
therefore $\chi_{\delta_{n}^{\epsilon}}(y)=\epsilon\sqrt{(-1)^n}$.
The fusion rule associated with the embedding $\pi=\delta_1$ is seen as:
$\pi \otimes\delta_{0}^{\epsilon}=\delta_1$,
 $\pi \otimes\delta_i=\delta_{i-1}\oplus\delta_{i+1}(1\leq i\leq n-1)$, $\pi \otimes\delta_{n}^{\epsilon}=\delta_{n-1}$.
Therefore the representation graph $\mathcal R_{\delta_1}(D_n)$ realizes the Dynkin diagram of type ${\rm D}_{n+2}^{(1)}$. For reference
the character table of $D_{n}$ is given in $\mathrm{Table}$ $1$ (also see \cite{Blu,Mon,St}),
where $|C_{G}(g)|$ is the cardinality of the conjugacy class containing $g$ in $G$.

\hskip6pt

\begin{center}
\begin{tikzpicture}
\matrix (m) [matrix of math nodes, row sep=4mm, column sep=0.8cm]
   { & \delta_{0}^- &  & & \delta_{n}^- & \\
   \delta_{0}^+ & \delta_{1} & \delta_{2} & \cdots & \delta_{n-1} & \delta_{n}^+ \\};
\path
(m-2-1) edge (m-2-2) (m-2-2) edge (m-2-3)  (m-2-5) edge (m-2-6)
(m-1-2) edge (m-2-2) (m-1-5) edge (m-2-5);
\draw[line width=0.4pt] (-0.65,-0.55)--(0.2,-0.55) (0.85,-0.55)--(1.65,-0.55);
\end{tikzpicture}
\end{center}

\hskip6pt

\ \  $\textbf{Table 1}$
\ \ \ \ \ \ \ \ \ \ \ \ \ \ \ \ \ \ \ \ \ \ \ \ \ \ \ \ \ \ \ \ \ \ \ \ \ \ \ \ \
\ \ \ \ \ \ \ \ \ \ \ \ \ \ \ \ \ \ \ \ \ \ \
The character table of $D_n$.

\ \  \begin{tabular}{l@{\indent\indent}c@{\indent\indent}
  c@{\indent\indent}c@{\indent\indent}
  c@{\indent\indent}c@{\indent\indent}c@{\indent\indent}c@{\indent}}
  \hline
   $\chi \backslash g$ & $1$ & $-1$ & $x$ & $\cdots$ & $x^{n-1}$ & $y$& $yx$ \\
   $ \backslash|C_{G}(g)|$ & $1$ & $1$ & $2$   & $\cdots$ & $2$ & $n$ & $n$ \\
  \hline
  $\chi_{\delta_{0}^{\pm}}$ & $1$ & $1$ & $1$ & $\cdots$ & $1$ & $\pm 1$ & $\pm 1$ \\
  $\chi_{\delta_i}  $ & $2$ & $2(-1)^i$ & $\theta_{2n}^i+\theta_{2n}^{-i}$ & $\cdots$ & $\theta_{2n}^{i(n-1)}+\theta_{2n}^{-i(n-1)}$ & $0$ & $0$ \\
  $\chi_{\delta_{n}^{\pm}}$ & $1$ & $(-1)^n$ & $-1$ & $\cdots$ & $(-1)^{n-1}$ & $\pm\sqrt{(-1)^n}$ & $\mp\sqrt{(-1)^n}$ \\

  \hline
\end{tabular}

\bigskip

\subsubsection{Binary tetrahedral group}\label{s:tetra}

The binary tetrahedral group $T$ of order $24$ is
$\langle x,y, z | x^2=y^2=z^{3}=-1, yxy^{-1}=x^{-1},
zxz^{-1}=y^{-1}x,zyz^{-1}=x\rangle$,
and the elements $\pm1, x, \pm z$ and $\pm z^2$ form
a set of representatives of the seven conjugacy classes.
The embedding $\pi$ of $T$ into $\mathrm{SU}_2$ is given by
\begin{align}\label{tg}
  \pi(x)=\left(
      \begin{array}{cc}
       \sqrt{-1}& 0 \\
        0 & -\sqrt{-1}\\
      \end{array}
    \right), \qquad
  \pi(y)=\left(
      \begin{array}{cc}
        0 &\sqrt{-1}\\
       \sqrt{-1}& 0 \\
      \end{array}
    \right), \qquad
  \pi(z)=\frac{1}{\sqrt{2}}\left(
                        \begin{array}{cc}
                          \theta_8^{-1} & \theta_8^{-1} \\
                          -\theta_8 & \theta_8 \\
                        \end{array}
                      \right).
\end{align}

The binary dihedral group $D_2=\langle x, y\rangle$ is a normal subgroup of $T$ of index $3$. Though $D_2$ has $5$ conjugacy classes, they
generate only $3$ conjugacy classes in $T$. Therefore the irreducible $D_2$-modules are induced to only three distinct $T$-modules such that
$\hat\delta_{2}^{+}=\hat\delta_{0}^{-}=\hat\delta_{2}^{-}$, more explicitly the induced characters obey that
\begin{align}\label{e:charT1}
   & \chi_{\hat\delta_{i}^{\pm}}(\pm 1)=3\chi_{\delta_{i}^{\pm}} (\pm 1)=3 \ \ (i=0,2), \ \ \ \
   \chi_{\hat\delta_1}(\pm 1)=3\chi_{\delta_1}(\pm 1)=\pm 6,\\ \label{e:charT2}
   & \chi_{\hat\delta_{0}^{+}}(x)=3\chi_{\delta_{0}^{+}} (x)=3, \ \ \chi_{\hat\delta_{2}^{\pm}}(x)=\chi_{\hat\delta_{0}^{-}}(x)=\chi_{\delta_{2}^{\pm}}(x)=-1,
   \ \  \chi_{\hat\delta_1}(x)=\chi_{\delta_1}(x)=0,\\ \label{e:charT3}
   & \chi_{\hat\delta_i^{\pm}} (\pm z)=\chi_{\hat\delta_i^{\pm}}(\pm z^2)=0  \ \ (i=0,2), \ \ \ \  \chi_{\hat\delta_1} (\pm z)=\chi_{\hat\delta_1} (\pm z^2)=0 .
\end{align}

Explicitly these three induced $T$-modules decompose into seven irreducible $T$-modules as follows:
\begin{align*}
   & \hat\delta_{0}^{+}=\tau_{0} \oplus \tau_{0}' \oplus \tau_{0}'', \\
   & \hat\delta_1=\tau_{1} \oplus \tau_{1}'\oplus \tau_{1}'', \\
   & \hat\delta_{2}^{+}=\hat\delta_{0}^{-}=\hat\delta_{2}^{-}:=\tau_2,
\end{align*}
where $\tau_0$ is the trivial module and $\{\tau_0, \tau_0', \tau_0'', \tau_1, \tau_1', \tau_1'', \tau_2\}$ forms the complete set of irreducible $T$-characters. It follows from \eqref{e:charT1}-\eqref{e:charT3} that $\pi\simeq \tau_{1}$ and
\begin{align*}
   & \chi_{\tau_{0}'}(\pm 1)=\chi_{\tau_{0}''}(\pm 1)=\chi_{\tau_{0}'}(x)=\chi_{\tau_{0}''}(x)=1,  \\
   & \chi_{\tau_{0}'}(\pm z)=\chi_{\tau_{0}''}(\pm z^2)=\theta_3, \\
   & \chi_{\tau_{0}'}(\pm z^2)=\chi_{\tau_{0}''}(\pm z)=\theta_3^2,\\
   & \chi_{\tau_{1}'}(\pm 1)=\chi_{\tau_{1}''}(\pm 1)=\pm 2, \ \ \chi_{\tau_{1}'}(x)=\chi_{\tau_{1}''}(x)=0,   \\
   & \chi_{\tau_{1}'}(\pm z)=\chi_{\tau_{1}''}(\mp z^2)=\pm\theta_3, \\
   & \chi_{\tau_{1}'}(\pm z^2)=\chi_{\tau_{1}''}(\mp z)=\mp\theta_3^2.
\end{align*}
Then $\pi\otimes \tau_{0}^{(i)}=\tau_{1}^{(i)}$,
$\pi\otimes \tau_{1}^{(i)}=\tau_{0}^{(i)}\oplus \tau_{2}(0\leq i\leq 2)$,
$\pi\otimes \tau_2 =\oplus_{i=0}^2\tau_{1}^{(i)}$
give rise to the Dynkin diagram of type ${\rm E_6^{(1)}}$. The
character table of $T$ is given in $\mathrm{Table}$ $2$ (also see \cite{Blu},  \cite{St}).

\hskip6pt

\begin{center}
\begin{tikzpicture}
\matrix (m) [matrix of math nodes, row sep=3.5mm, column sep=0.9cm]
   { &  & \tau_{0}' & &   \\
   &  & \tau_{1}' & &  \\
   \tau_{0} & \tau_{1} & \tau_{2} & \tau_{1}'' & \tau_{0}'' \\};
\path
(m-3-1) edge (m-3-2)  (m-3-4) edge (m-3-5)
(m-1-3) edge (m-2-3) (m-2-3) edge (m-3-3);
\draw[line width=0.4pt] (-1.25,-0.98)--(-0.3,-0.98) (0.2,-0.98)--(1.1,-0.98);
\end{tikzpicture}
\end{center}

\hskip6pt

\ \ \ \ \  $\textbf{Table 2}$
\ \ \ \ \ \ \ \ \ \ \ \ \ \ \ \ \ \ \ \ \ \ \ \ \ \ \ \ \ \ \ \ \ \ \ \ \ \ \ \ \
\ \ \ \ \ \ \ \ \ \ \ \ \ \ \ \ \ \ \ \ \ \ \
The character table of $T$.

\ \ \ \ \  \begin{tabular}{l@{\indent\indent\indent}c@{\indent\indent\indent\indent}c@{\indent\indent\indent\indent}
  c@{\indent\indent\indent\indent}c@{\indent\indent\indent\indent}
  c@{\indent\indent\indent\indent}c@{\indent\indent\indent\indent}c@{\indent}}
  \hline
   $\chi \backslash g$ & $1$ & $-1$ & $x$ & $z$ & $z^2$ & $-z$ & $-z^2$ \\
   $\ \ \ \backslash|C_{G}(g)|$ & $1$ & $1$ & $6$ & $4$ & $4$ & $4$ & $4$ \\
  \hline
  $\chi_{\tau_{0}}$ & $1$ & $1$ & $1$ & $1$ & $1$ & $1$ & $1$ \\
  $\chi_{\tau_{0}'}$ & $1$ & $1$ & $1$ & $\theta_3$ & $\theta_3^2$ & $\theta_3$ & $\theta_3^2$ \\
  $\chi_{\tau_{0}''}$ & $1$ & $1$ & $1$ & $\theta_3^2$ & $\theta_3$ & $\theta_3^2$ & $\theta_3$ \\
  $\chi_{\tau_{1}}$ & $2$ & $-2$ & $0$ & $1$ & $-1$ & $-1$ & $1$ \\
  $\chi_{\tau_{1}'}$ & $2$ & $-2$ & $0$ & $\theta_3$ & $-\theta_3^2$ & $-\theta_3$ & $\theta_3^2$ \\
  $\chi_{\tau_{1}''}$ & $2$ & $-2$ & $0$ & $\theta_3^2$ & $-\theta_3$ & $-\theta_3^2$ & $\theta_3$ \\
  $\chi_{\tau_2}$ & $3$ & $3$ & $-1$ & $0$ & $0$ & $0$ & $0$ \\
  \hline
\end{tabular}

\bigskip

\subsubsection{Binary octahedral group}\label{s:octa}

Let $O=\langle u,y, z |u^4=y^2=z^{3}=-1, yuy^{-1}=u^{-1},
zuz^{-1}=z^{-1}u^{-1},zyz^{-1}=u^2\rangle$
be the binary octahedral group order $48$. Then the subgroup $\langle u^2, y, z\rangle\simeq T$,
the binary tetrahedral group $T$ of order $24$ (see
subsection 2.1.3), so the former is also denoted by $T$. The group $O$ is imbedded into $\mathrm{SU}_2$ by
letting
$\pi(u)=diag(\theta_{8}, \theta_{8}^{-1})$ and $\pi(y), \pi(z)$ given by
\eqref{tg}.

The subgroup $T$ has $7$ conjugacy classes but generates $5$ conjugacy classes in $O$, then the $7$ irreducible
$T$-modules $\tau_{0}^{(i)}, \tau_{1}^{(i)}, \tau_2$ $(i=0,1,2)$ are induced into
$5$ different $O$-modules as $\hat\tau_{1}'=\hat\tau_{1}'', \hat\tau_{0}'=\hat\tau_{0}''$, and the induced characters are
computed by
$$\chi_{\hat\tau_k^{(i)}}(g)=\begin{cases} \chi_{\tau_k^{(i)}}(g)+\chi_{\tau_k^{(i)}}(h^{-1}gh), & \forall g,h\in T\\
   0, &   \forall g \in O\setminus T\end{cases}$$
for $k=0,1, i=0,1,2$ and $k=2, i=0$. Therefore the induced modules decompose into irreducible $O$-modules as follows:
\begin{align*}
   & \hat\tau_{0}=\omega_0^{+}\oplus\omega_0^-, \ \ \hat\tau_{1}=\omega_1^+\oplus\omega_1^-, \ \  \hat\tau_2=\omega_2^+\oplus\omega_2^-, \\
   & \hat\tau_{1}'=\hat\tau_{1}'':=\omega_3, \ \ \hat\tau_{0}'=\hat\tau_{0}'':=\omega_4.
\end{align*}
and the irreducible summands $\{\omega_i^{(\pm)}, \omega_3, \omega_4| i=0, 1, 2\}$ realize all irreducible $O$-modules.

Clearly $\omega_0^+$ is the trivial module and $\pi_1=\omega_1^+$ (resp. $\omega_1^-$), and we record other character values in the following
\begin{align*}
   \chi_{\omega_0^-}(\pm 1)&=\chi_{\omega_0^-}(y)=\chi_{\omega_0^-}(\pm z)=1, \ \ \chi_{\omega_0^-}(\pm u)=\chi_{\omega_0^-}(uy)=-1, \\
   \chi_{\omega_2^{\pm}}(\pm 1)&=3, \ \ \ \chi_{\omega_2^{\pm}}(y)=-1, \ \ \ \chi_{\omega_2^{\pm}}(\pm z)=0, \\
   \chi_{\omega_2^+}(\pm u)&=\chi_{\omega_2^-}(uy)=1, \ \ \chi_{\omega_2^-}(\pm u)=\chi_{\omega_2^+}(uy)=-1.
\end{align*}
Therefore,
$\pi\otimes \omega_{0}^{\pm} = \omega_{1}^{\pm}$,
$\pi\otimes \omega_{1}^{\pm} \simeq \omega_{0}^{\pm}\oplus \omega_{2}^{\pm}$,
$\pi\otimes \omega_{2}^{\pm} \simeq \omega_{1}^{\pm} \oplus \omega_{3}$,
$\pi\otimes \omega_{3}=\omega_{2}^{\pm}\oplus\omega_{4}$,
$\pi\otimes \omega_{4}=\omega_{3}$, and
the Dynkin diagram of type ${\rm E_7^{(1)}}$ is realized by $T$-irreducible modules.
The character table of $O$ is given in $\mathrm{Table}$ $3$ (also see \cite{Blu}, \cite{St}).

\hskip6pt

\begin{center}
\begin{tikzpicture}
\matrix (m) [matrix of math nodes, row sep=3.5mm, column sep=0.9cm]
   { &  &  & \omega_4 & &  &  \\
   \omega_0^+ & \omega_1^+ & \omega_2^+ & \omega_3 & \omega_2^- & \omega_1^- & \omega_0^- \\};
\path
(m-2-1) edge (m-2-2) (m-2-2) edge (m-2-3)
(m-2-5) edge (m-2-6) (m-2-6) edge (m-2-7)
(m-1-4) edge (m-2-4);
\draw[line width=0.4pt] (-1.25,-0.4)--(-0.3,-0.4) (0.3,-0.4)--(1.2,-0.4);
\end{tikzpicture}
\end{center}

\hskip6pt

\ \ \ \ $\textbf{Table 3}$
\ \ \ \ \ \ \ \ \ \ \ \ \ \ \ \ \ \ \ \ \ \  \ \ \ \ \ \ \ \ \ \  \  \ \ \ \ \ \ \ \ \ \ \ \ \ \ \ \  \  \ \ \ \ \ \ \ \ \ \
The character table of $O$.

\ \ \ \ \begin{tabular}{l@{\indent\indent\indent}c@{\indent\indent\indent}c@{\indent\indent\indent}
  c@{\indent\indent\indent}c@{\indent\indent\indent}c@{\indent\indent\indent}
  c@{\indent\indent\indent}c@{\indent\indent\indent}c@{\indent}}
  \hline
  $\chi \backslash g$ & $1$ & $-1$  & $u$ & $-{u}$ & $y$ & $uy$ & $z$ & $-z$ \\
   $\ \ \ \backslash|C_{G}(g)|$ & $1$ & $1$ & $6$ & $6$ & $6$ & $12$ & $8$ & $8$  \\
  \hline
  $\chi_{\omega_0^+}$ & $1$ & $1$ & $1$  & $1$ & $1$ & $1$ & $1$ & $1$ \\
  $\chi_{\omega_1^+}$ & $2$ & $-2$ & $\sqrt{2}$ & $-\sqrt{2}$ & $0$ & $0$  & $1$ & $-1$ \\
  $\chi_{\omega_2^+}$ & $3$ & $3$  & $1$ & $1$  & $-1$ & $-1$ & $0$ & $0$ \\
  $\chi_{\omega_{3}}$ & $4$ & $-4$  & $0$ & $0$ & $0$ & $0$  & $-1$ & $1$ \\
  $\chi_{\omega_{4}}$ & $2$ & $2$ & $0$ & $0$ & $2$ & $0$  & $-1$ & $-1$ \\
  $\chi_{\omega_2^-}$ & $3$ & $3$  & $-1$ & $-1$ & $-1$ & $1$  & $0$ & $0$ \\
  $\chi_{\omega_1^-}$ & $2$ & $-2$ & $-\sqrt{2}$ & $\sqrt{2}$ & $0$ & $0$  & $1$ & $-1$ \\
  $\chi_{\omega_0^-}$ & $1$ & $1$ & $-1$ & $-1$ & $1$ & $-1$  & $1$ & $1$ \\
  \hline
\end{tabular}

~~~~~~~~~~~~~~~~~~~~~~~~~

\subsection{Realizations of ${\rm A_{2n-1}^{(2)}}$ and ${\rm B_{n}^{(1)}}$ by the pair ($D_{2(n-1)}$,$D_{n-1}$)} \label{e:AB}

\qquad

We now start to describe the McKay-Slodowy correspondence, which covers the affine Dynkin diagrams of non-simply laced types. The main idea is to use a pair of finite groups $N\vartriangleleft G$ to
realize the simple roots either as induced $G$-modules from the irreducible $N$-modules or $N$-restrictions of the irreducible $G$-modules.

For fixed $n\geq 3$, let $D_{2(n-1)}=\langle x, y\rangle$ be the binary dihedral group of order $8(n-1)$, where $x^{2(n-1)}=y^2=-1, yxy^{-1}=x^{-1}$.
So $\langle x^2, y\rangle=D_{n-1}\vartriangleleft D_{2(n-1)}$ and $|D_{2(n-1)}:D_{n-1}|=2$.
As explained in subsection \ref{s:binary}, there are $2n+1$ irreducible $D_{2(n-1)}$-modules:
$\delta_{0}^{\pm}, \delta_i , \delta_{2(n-1)}^{\pm}$ $(1\leq i\leq 2n-3)$, and $n+2$ irreducible $D_{n-1}$-modules:
${\delta'}_{0}^{\pm}, \delta'_i, {\delta'}_{{n-1}}^{\pm}$ $(1\leq i\leq n-2)$.

First we consider restriction of the irreducible $D_{2(n-1)}$-modules to the subgroup $D_{n-1}$.
Note that $\chi_{\delta_i}(x)=\theta_{4(n-1)}^{i}+\theta_{4(n-1)}^{-i}$ for $1\leq i\leq 2n-3$,
$\chi_{\delta_0^{\pm}}(y)=\chi_{\delta_0^{\pm}}(yx)=\pm 1$
and $\chi_{\delta_{2(n-1)}^{\pm}}(y)=\chi_{\delta_{2(n-1)}^{\mp}}(yx)=\pm 1$
(subsection \ref{s:binary}). Therefore there are only $n+1$ restrictions and they satisfy the following relations:
\begin{align}\label{e:res-dihedral1}
  &\check{\delta}_{0}^{\pm}=\check{\delta}_{2(n-1)}^{\pm}={\delta'}_{0}^{\pm}, \qquad  \ \ \ \ \ \ \ \
  \check{\delta}_{i}=\check{\delta}_{2(n-1)-i}=\delta'_{i} \ \ (i=1,2,\ldots,n-2),  \\ \label{e:res-dihedral2}
  &\check{\delta}_{n-1}={\delta'}_{n-1}^{+}\oplus {\delta'}_{n-1}^{-}.
\end{align}

On the other hand, the induced characters of $\chi_{\delta'_i}$ can be written as
\begin{equation}\label{rifor}
  \chi_{\hat{\delta}'_k} =\sum\limits_{i=0}^{2(n-1)}(\chi_{\delta_i} , \chi_{\hat{\delta}'_k} )_{G}\chi_{\delta_i}
  =\sum\limits_{i=0}^{2(n-1)}(\chi_{\check{\delta}_i} , \chi_{\delta'_k} )_{H}\chi_{\delta_i} ,
\end{equation}
where $(\chi_{\delta_i} , \chi_{\hat{\delta}'_k} )_{G}=(\chi_{\check{\delta}_i} , \chi_{\delta'_k} )_{H}$
by the Frobenius reciprocity (see \cite{Kar}). In view of \eqref{e:res-dihedral1}-\eqref{e:res-dihedral2} the equation
\eqref{rifor} implies that:
\begin{align*}
  & \hat{\delta'}_{0}^{\pm}=\delta_{0}^{\pm}\oplus\delta_{2(n-1)}^{\pm}, \qquad \  \ \ \ \ \ \ \  \hat{\delta'}_i =\delta_{i}\oplus\delta_{2(n-1)-i} \ \ (i=1,2,\ldots,n-2), \\ 
  & \hat{\delta'}_{n-1}^{+}=\hat{\delta'}_{n-1}^{-}=\delta_{n-1}. 
\end{align*}

Using the imbedding $\pi=\delta'_1$ for the set $\{\check{\delta}_0^{\pm}, \check{\delta}_i |1\leq i\leq n-1\}$ and
$\pi=\delta_1$ for the set $\{\hat{\delta'}_0^{\pm}, \hat{\delta'}_i, \hat{\delta'}_{n-1}^{+}|1\leq i\leq n-2\}$, the
fusion rules are easily computed as follows (by using those of ${\rm D}_{2n}^{(1)}$ or ${\rm D}_{n+1}^{(1)}$):
\begin{align*}
  &\delta_1'\otimes \check\delta_{0}^{\pm}=\check\delta_1, \qquad \qquad  \qquad \qquad  \qquad  \ \ \ \ \ \ \ \
  \delta_{1}\otimes \hat{\delta'}_{0}^{+}=\delta_{1}' \otimes \hat{\delta'}_{0}^{-} = \hat{\delta'}_1,   \\
  &\delta_1'\otimes \check\delta_1=\check\delta_{0}^+\oplus\check\delta_{0}^-\oplus\check\delta_2, \qquad \qquad \qquad  \ \ \ \ \ \
   \delta_{1}\otimes \hat{\delta'}_1=\hat{\delta'}_{0}^{+}\oplus\hat{\delta'}_{0}^{-}\oplus\hat{\delta'}_2, \\
  &\delta_1'\otimes \check\delta_i=\check\delta_{i-1}\oplus\check\delta_{i+1} \ (2\leq i\leq n-2), \ \ \ \ \ \
   \delta_{1}\otimes \hat{\delta'}_i =\hat{\delta'}_{i-1}\oplus\hat{\delta'}_{i+1} \ (2\leq i\leq n-3),  \\
  &\delta_1'\otimes \check\delta_{n-1}=2\check\delta_{n-2}, \qquad \qquad  \qquad \qquad   \ \ \ \ \ \
  \delta_{1}\otimes \hat{\delta'}_{n-2}=\hat{\delta'}_{n-3}\oplus 2\hat{\delta'}_{n-1}^{+},   \\
  &\hskip 6.3cm
  \delta_{1}\otimes \hat{\delta'}_{n-1}^{+}=\hat{\delta'}_{n-2}.
\end{align*}
The corresponding Dynkin diagrams are ${\rm A}_{2n-1}^{(2)}$ (i.e. $\mathcal{R}_{\delta_1'}(\check{D}_{2(n-1)})$) and ${\rm B}_n^{(1)}$
(i.e. $\mathcal{R}_{\delta_1}(\hat{D}_{n-1})$) respectively.
The numbers inside the nodes are the degrees of the characters.

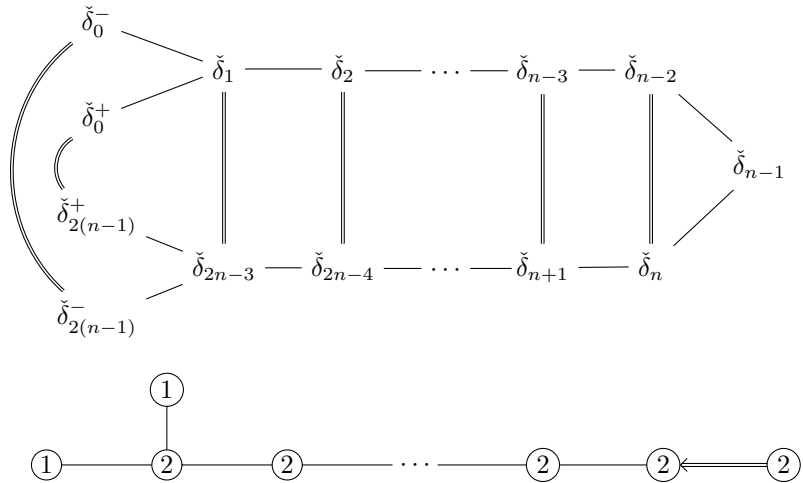
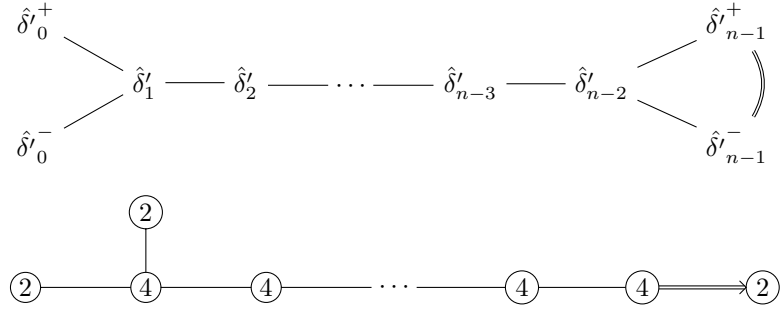
\begin{center}
\begin{tikzpicture}
\matrix (m) [matrix of math nodes, row sep=0.0005mm, column sep=0.5cm]
    {    \check{\delta}_{0}^-  &  &   &  &   &  & \\
           &  \check{\delta}_1 & \check{\delta}_{2}  &
       \cdots & \check{\delta}_{n-3}  & \check{\delta}_{n-2} & \\
         \check{\delta}_{0}^+ &  &   &  &   &  & \\
          &   &   &  &   &  &  \check{\delta}_{n-1} \\
          \check{\delta}_{2(n-1)}^+  &   &   &  &   &  &  \\
          & \check{\delta}_{2n-3} & \check{\delta}_{2n-4}  &
       \cdots &  \check{\delta}_{n+1} & \check{\delta}_{n} & \\
       \check{\delta}_{2(n-1)}^-   &   &   &  &   &  &  \\};
\path
(m-1-1) edge (m-2-2) (m-3-1) edge (m-2-2) (m-7-1) edge (m-6-2) (m-5-1) edge (m-6-2)
(m-2-2) edge (m-2-3) (m-2-5) edge (m-2-6) (m-2-6) edge (m-4-7)
(m-6-2) edge (m-6-3) (m-6-5) edge (m-6-6) (m-6-6) edge (m-4-7)
(m-2-3) edge[double] (m-6-3) (m-2-2) edge[double] (m-6-2)
(m-2-6) edge[double] (m-6-6) (m-2-5) edge[double] (m-6-5);
\draw[-][double] (m-1-1) to [bend right=50] (m-7-1);
\draw[-][double] (m-3-1) to [bend right=50] (m-5-1);
\draw[line width=0.4pt] (-0.75,1.35)--(0,1.35) (0.65,1.35)--(1.15,1.35);
\draw[line width=0.4pt] (-0.5,-1.28)--(0,-1.28) (0.65,-1.28)--(1.2,-1.28);
\end{tikzpicture}
\end{center}

\begin{center}
\begin{tikzpicture}
\draw node{} node at(2.2,1){$1$} (2.2,1)circle[radius=0.22] (2.2,0.78)--(2.2,0.2);
\draw node{} node at(0.6,0){$1$} (0.6,0)circle[radius=0.2] (0.8,0)--(2,0);
\draw node at(2.2,0){$2$} (2.2,0)circle[radius=0.2] (2.4,0)--(3.6,0);
\draw  node at(3.8,0){$2$} (3.8,0)circle[radius=0.2](4,0)--(5.2,0);
\draw  node at(5.53,0){$\cdots$} (5.8,0)--(6.98,0);
\draw node{} node at(7.2,0){$2$} (7.2,0)circle[radius=0.22] (7.42,0)--(8.58,0);
\draw node at(8.8,0){$2$} (8.8,0)circle[radius=0.22];
\draw [->,>=angle 90] (9.03,0)--(9.02,0);
\draw (9.04,0.02)--(10.18,0.02) (9.04,-0.02)--(10.18,-0.02);
\draw node at(10.4,0){$2$} (10.4,0)circle[radius=0.22];
\end{tikzpicture}
\end{center}

\ \ \ \ \ \ \ \ \ \ \ \ \ \ \ \ \ \ \ \

\begin{center}
\ \ \ \ \ \begin{tikzpicture}
\matrix (m) [matrix of math nodes, row sep=0.1cm, column sep=0.8cm]
    {    \hat{\delta'}_{0}^{+} &  &   &  &   &  & \hat{\delta'}_{{n-1}}^{+} \\
           &  \hat{\delta}_{1}' & \hat{\delta}_{2}'  &
       \cdots & \hat{\delta}_{n-3}'  & \hat{\delta}_{n-2}' & \\
        \hat{\delta'}_{0}^{-} &   &   &  &   &  & \hat{\delta'}_{{n-1}}^{-} \\};
\path
(m-1-1) edge (m-2-2) (m-3-1) edge (m-2-2) (m-1-7) edge (m-2-6) (m-3-7) edge (m-2-6)
(m-2-2) edge (m-2-3)
(m-2-5) edge (m-2-6);
\draw[-][double] (m-1-7) to [bend right=-30] (m-3-7);
\draw[line width=0.4pt] (-1.65,-0.02)--(-0.85,-0.02) (-0.25,-0.02)--(0.55,-0.02);
\end{tikzpicture}
\end{center}

\begin{center}
\begin{tikzpicture}
\draw node{} node at(2.2,1){$2$} (2.2,1)circle[radius=0.22] (2.2,0.78)--(2.2,0.2);
\draw node{} node at(0.6,0){$2$} (0.6,0)circle[radius=0.2] (0.8,0)--(2,0);
\draw node at(2.2,0){$4$} (2.2,0)circle[radius=0.2] (2.4,0)--(3.6,0);
\draw  node at(3.8,0){$4$} (3.8,0)circle[radius=0.2](4,0)--(5.2,0);
\draw  node at(5.53,0){$\cdots$} (5.8,0)--(6.98,0);
\draw node{} node at(7.2,0){$4$} (7.2,0)circle[radius=0.22] (7.42,0)--(8.58,0);
\draw node at(8.8,0){$4$} (8.8,0)circle[radius=0.22];
\draw [->,>=angle 90] (10.18,0)--(10.19,0);
\draw (9.02,0.02)--(10.16,0.02) (9.02,-0.02)--(10.16,-0.02);
\draw node at(10.4,0){$2$} (10.4,0)circle[radius=0.22];
\end{tikzpicture}
\end{center}

\subsection{Realizations of ${\rm D_{n+1}^{(2)}}$ and ${\rm C_{n}^{(1)}}$ by the pair ($D_{n}$,$C_{2n}$)}\label{s:typeDC}

The cyclic group $C_{2n}=\langle x\rangle$ is a normal subgroup of $D_n=\langle x, y\rangle$ with index $2$.
The restriction $\chi_{\check{\delta}_i}(x)=\theta_{2n}^i+\theta_{2n}^{-i}=\chi_{\xi_i}(x)+\chi_{\xi_{2n-i}}(x)$, where
$\chi_{\xi_i}$ are irreducible characters of $C_{2n}$, then
$\check{\delta}_i=\xi_i+\xi_{2n-i}$ for $1\leq i\leq n-1$. Dually one has that
$\hat{\xi}_i=\delta_i$ for $1\leq i\leq n-1$. In this way, one obtains that
\begin{align*}
  &\check{\delta}_{0}^{\pm}=\xi_{0},  \ \ \ \  \ \ \ \ \ \ \ \ \ \ \check{\delta}_i=\xi_i\oplus\xi_{2n-i}\ (i=1,2,\ldots,n-1),
   \ \ \ \  \check{\delta}_{n}^{\pm}=\xi_{n},  \\
  &\hat{\xi}_0=\delta_{0}^+\oplus\delta_{0}^-, \ \ \ \ \   \hat{\xi}_i=\hat{\xi}_{2n-i}=\delta_i\ (i=1,2,\ldots,n-1),
  \ \ \ \    \hat{\xi}_{n}=\delta_{n}^+\oplus\delta_{n}^-.
\end{align*}

Note that the embedding in this case is taken as either
$\xi_1\oplus\xi_{-1}=\check{\delta}_1$ or $\hat{\xi}_1=\delta_1$. If follows from the fusion rules of $C_{2n}$ and $D_n$ modules that
\begin{align*}
   & (\xi_1\oplus\xi_{-1})\otimes \check{\delta}_{0}^+ =\check{\delta}_1, \qquad \ \ \ \ \ \ \ \ \ \ \ \ \ \ \ \ \ \ \ \
     \delta_{1}\otimes \hat{\xi}_{0}=2\hat{\xi}_1,\\
   & (\xi_1\oplus\xi_{-1})\otimes \check{\delta}_1=2\check{\delta}_{0}^+ \oplus \check{\delta}_{2}, \qquad \ \ \ \ \ \ \ \ \
    \delta_{1}\otimes \hat{\xi}_{1}=\hat{\xi}_0\oplus \hat{\xi}_2,\\
   &(\xi_1\oplus\xi_{-1})\otimes \check{\delta}_i =\check{\delta}_{i-1}\oplus \check{\delta}_{i+1}, \qquad \ \ \ \ \ \
    \delta_{1}\otimes \hat{\xi}_i=\hat{\xi}_{i-1}\oplus \hat{\xi}_{i+1} \ (2\leq i\leq n-2), \\
   & (\xi_1\oplus\xi_{-1})\otimes \check{\delta}_{n-1}=\check{\delta}_{n-2}\oplus 2\check{\delta}_{n}^+, \qquad
     \delta_{1}\otimes \hat{\xi}_{n-1}=\hat{\xi}_{n-2}\oplus \hat{\xi}_{n}, \\
   & (\xi_1\oplus\xi_{-1})\otimes \check{\delta}_{n}^+= \check{\delta}_{n-1}, \qquad \ \ \ \ \ \ \ \ \ \ \ \ \ \ \
    \delta_{1}\otimes \hat{\xi}_{n}=2\hat{\xi}_{n-1}.
\end{align*}
Therefore the representation graph $\mathcal{R}_{(\xi_1+\xi_{-1})}(\check{D}_{n})$ realizes the Dynkin diagram ${\rm D_{n+1}^{(2)}}$
and $\mathcal{R}_{\delta_1}(\hat{C}_{2n})$ realizes the Dynkin diagram ${\rm C_n^{(1)}}$, which are depicted as follows.

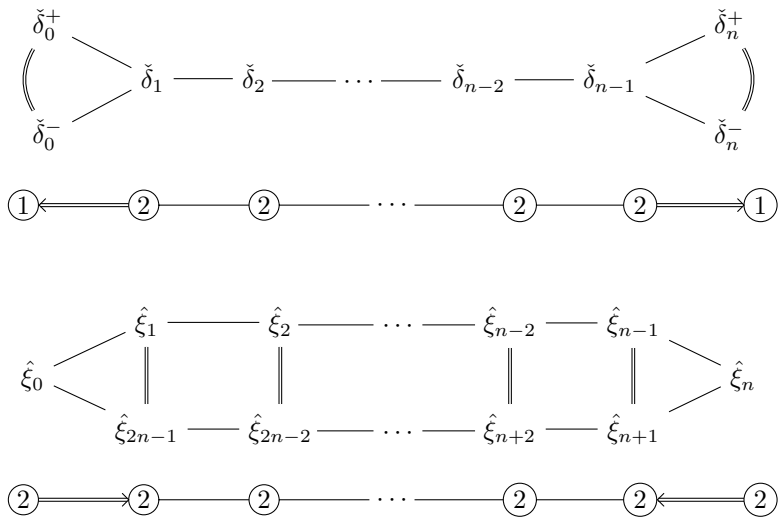
\begin{center}
\begin{tikzpicture}
\matrix (m) [matrix of math nodes, row sep=0.1cm, column sep=0.8cm]
    {    \check{\delta}_{0}^+ &  &   &  &   &  & \check{\delta}_{n}^+ \\
           &  \check{\delta}_{1} & \check{\delta}_{2}  &
       \cdots & \check{\delta}_{n-2}  & \check{\delta}_{n-1} & \\
        \check{\delta}_{0}^- &   &   &  &   &  & \check{\delta}_{n}^- \\};
\path
(m-1-1) edge (m-2-2) (m-3-1) edge (m-2-2) (m-1-7) edge (m-2-6) (m-3-7) edge (m-2-6)
(m-2-2) edge (m-2-3) (m-2-5) edge (m-2-6);
\draw[-][double] (m-1-1) to [bend right=30] (m-3-1);
\draw[-][double] (m-1-7) to [bend right=-30] (m-3-7);
\draw[line width=0.4pt] (-1.55,-0.015)--(-0.7,-0.015) (-0.1,-0.015)--(0.7,-0.015);
\end{tikzpicture}
\end{center}

\begin{center}
\begin{tikzpicture}
\draw node at(0.2,0){$1$} (0.2,0)circle[radius=0.2];
\draw [->,>=angle 90] (0.41,0)--(0.4,0);
\draw (0.42,0.02)--(1.6,0.02) (0.42,-0.02)--(1.6,-0.02);
\draw node at(1.8,0){$2$} (1.8,0)circle[radius=0.2] (2,0)--(3.2,0);
\draw  node at(3.4,0){$2$} (3.4,0)circle[radius=0.2](3.6,0)--(4.8,0);
\draw  node at(5.13,0){$\cdots$} (5.4,0)--(6.58,0);
\draw node{} node at(6.8,0){$2$} (6.8,0)circle[radius=0.22] (7.02,0)--(8.18,0);
\draw node at(8.4,0){$2$} (8.4,0)circle[radius=0.22];
\draw [->,>=angle 90] (9.78,0)--(9.79,0);
\draw (8.62,0.02)--(9.76,0.02) (8.62,-0.02)--(9.76,-0.02);
\draw node at(10,0){$1$} (10,0)circle[radius=0.22];
\end{tikzpicture}
\end{center}

\ \ \ \ \ \ \ \ \ \ \ \ \ \ \ \ \ \ \

\begin{center}
\begin{tikzpicture}
\matrix (m) [matrix of math nodes, row sep=0.06cm, column sep=0.7cm]
    {     & \hat{\xi}_{1} & \hat{\xi}_{2}  & \cdots & \hat{\xi}_{n-2}  & \hat{\xi}_{n-1} &   \\
        \hat{\xi}_{0}   &   &   &  &    &   & \hat{\xi}_{n} \\
          & \hat{\xi}_{2n-1}  & \hat{\xi}_{2n-2}  & \cdots  & \hat{\xi}_{n+2}  & \hat{\xi}_{n+1} & \\};
\path
(m-2-1) edge (m-1-2) (m-1-2) edge (m-1-3)
(m-1-5) edge (m-1-6) (m-1-6) edge (m-2-7)
(m-2-1) edge (m-3-2) (m-3-2) edge (m-3-3)
(m-3-5) edge (m-3-6) (m-3-6) edge (m-2-7)
(m-1-2) edge[double] (m-3-2) (m-1-3) edge[double] (m-3-3)
(m-1-5) edge[double] (m-3-5) (m-1-6) edge[double] (m-3-6);
\draw[line width=0.4pt] (-1.2,0.68)--(-0.2,0.68) (0.45,0.68)--(1.15,0.68);
\draw[line width=0.4pt] (-0.95,-0.73)--(-0.2,-0.73) (0.45,-0.73)--(1.15,-0.73);
\end{tikzpicture}
\end{center}

\begin{center}
\begin{tikzpicture}
\draw node at(0.2,0){$2$} (0.2,0)circle[radius=0.2];
\draw [->,>=angle 90] (1.6,0)--(1.61,0);
\draw (0.4,0.02)--(1.58,0.02) (0.4,-0.02)--(1.58,-0.02);
\draw node at(1.8,0){$2$} (1.8,0)circle[radius=0.2] (2,0)--(3.2,0);
\draw  node at(3.4,0){$2$} (3.4,0)circle[radius=0.2](3.6,0)--(4.8,0);
\draw  node at(5.13,0){$\cdots$} (5.4,0)--(6.58,0);
\draw node{} node at(6.8,0){$2$} (6.8,0)circle[radius=0.22] (7.02,0)--(8.18,0);
\draw node at(8.4,0){$2$} (8.4,0)circle[radius=0.22];
\draw [->,>=angle 90] (8.63,0)--(8.62,0);
\draw (8.64,0.02)--(9.78,0.02) (8.64,-0.02)--(9.78,-0.02);
\draw node at(10,0){$2$} (10,0)circle[radius=0.22];
\end{tikzpicture}
\end{center}

\subsection{Realizations of ${\rm A_{2n}^{(2)}}$ and ${\rm C_{n}^{(1)}}$ by the pair ($D_{2n}$,$C_{2n}$)}

Fix $n\geq 2$, the binary dihedral group $D_{2n}=\langle x, y\rangle$ contains the normal cyclic group $C_{2n}=\langle x^2\rangle$ with index
$4$. Similar to subsection \ref{s:typeDC} the general restriction $\chi_{\check\delta_i}(x^2)=\theta_{4n}^{2i}+\theta_{4n}^{-2i}
=\theta_{2n}^i+\theta_{2n}^{-i}$ for $1\leq i\leq 2n-1$, i.e. $\check\delta_i=\xi_i+\xi_{2n-i}$, where
$\xi_i$ are irreducible modules of $C_{2n}$. As for induction, $\hat\xi_i=\delta_i+\delta_{2n-i}$ due to
$\langle x^2\rangle\vartriangleleft\langle x\rangle \vartriangleleft D_{2n}$. Detailed restriction and induction relations are given as follows:
\begin{align*}
  &\check{\delta}_{0}^{\pm}=\check{\delta}_{2n}^{\pm}=\xi_{0},  \ \ \ \ \ \ \ \ \ \ \ \ \ \ \ \ \ \ \ \ \ \ \ \ \ \
  \check{\delta}_i=\check{\delta}_{2n-i}=\xi_i\oplus\xi_{2n-i}\ (i=1,2,\ldots, n), \\ 
  & \hat\xi_{0}=\delta_{0}^+\oplus\delta_{0}^-\oplus\delta_{2n}^+\oplus\delta_{2n}^-,   \ \ \ \ \ \ \
   \hat\xi_i=\hat\xi_{2n-i}=\delta_i \oplus\delta_{2n-i}\ (i=1,2,\ldots,n). 
\end{align*}
Using the embeddings $\check\delta_1=\xi_1\oplus\xi_{-1}: C_{2n}\hookrightarrow \mathrm{SU}_2$ and
$\delta_1: D_{2n}\hookrightarrow \mathrm{SU}_2$, one has the following fusion relations:
\begin{align*}
   & (\xi_1\oplus\xi_{-1})\otimes \check{\delta}_{0}^+=\check{\delta}_1,\ \ \ \ \ \ \qquad \qquad \ \ \ \ \ \
      \delta_1\otimes\hat\xi_0=2\hat\xi_1,   \\
   & (\xi_1\oplus\xi_{-1})\otimes \check{\delta}_1=2\check{\delta}_{0}^+\oplus\check{\delta}_{2},\ \ \ \ \ \ \qquad \ \ \ \
      \delta_1\otimes\hat\xi_1=\hat\xi_0\oplus\hat\xi_2,   \\
   &(\xi_1\oplus\xi_{-1})\otimes \check{\delta}_i=\check{\delta}_{i-1}\oplus\check{\delta}_{i+1}, \ \ \ \ \ \ \ \ \ \ \ \ \ \
      \delta_1\otimes\hat\xi_i=\hat\xi_{i-1}\oplus\hat\xi_{i+1} \  (2\leq i\leq n-2),   \\
   &(\xi_1\oplus\xi_{-1})\otimes \check{\delta}_{n-1}=\check{\delta}_{n-2}\oplus\check{\delta}_{n}, \ \ \ \qquad
      \delta_1\otimes\hat\xi_{n-1}=\hat\xi_{n-2}\oplus\hat\xi_{n},   \\
   & (\xi_1\oplus\xi_{-1})\otimes \check{\delta}_{n}=2\check{\delta}_{n-1}, \ \ \qquad \qquad \ \ \ \
      \delta_1\otimes\hat\xi_n=2\hat\xi_{n-1}.
\end{align*}
Therefore the representation graphs $\mathcal{R}_{(\xi_1+\xi_{-1})}(\check{D}_{2n})$ and $\mathcal{R}_{\delta_1}(\hat{C}_{2n})$
realize the twisted affine Dynkin diagram of type ${\rm A_{2n}^{(2)}}$ and the non-simply laced affine Dynkin diagram
of type ${\rm C_n^{(1)}}$, respectively. The exact relations are shown in the following diagrams, where the numbers
indicate the degrees of characters.

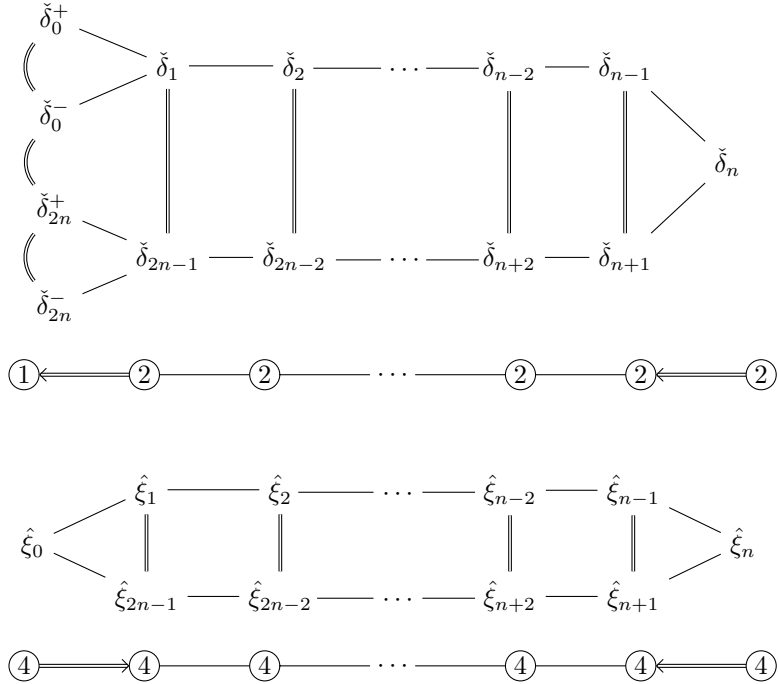
\begin{center}
\begin{tikzpicture}
\matrix (m) [matrix of math nodes, row sep=0.05mm, column sep=0.6cm]
    {   \check{\delta}_{0}^+  &   &   &  &   &  & \\
           &  \check{\delta}_1 & \check{\delta}_{2}  &
       \cdots & \check{\delta}_{n-2}  & \check{\delta}_{n-1} & \\
        \check{\delta}_{0}^- &   &   &  &   &  & \\
         &   &   &  &   &  &  \check{\delta}_{n} \\
       \check{\delta}_{2n}^+  &   &   &  &   &  &  \\
        & \check{\delta}_{2n-1} & \check{\delta}_{2n-2}  &
       \cdots &  \check{\delta}_{n+2} & \check{\delta}_{n+1} & \\
        \check{\delta}_{2n}^-  &   &   &  &   &  &  \\};
\path
(m-1-1) edge (m-2-2) (m-3-1) edge (m-2-2) (m-5-1) edge (m-6-2) (m-7-1) edge (m-6-2)
(m-2-2) edge (m-2-3) (m-2-5) edge (m-2-6) (m-2-6) edge (m-4-7)
(m-6-2) edge (m-6-3) (m-6-5) edge (m-6-6) (m-6-6) edge (m-4-7)
(m-2-2) edge[double] (m-6-2) (m-2-3) edge[double] (m-6-3)
(m-2-5) edge[double] (m-6-5) (m-2-6) edge[double] (m-6-6);
\draw[-][double] (m-1-1) to [bend right=40] (m-3-1);
\draw[-][double] (m-3-1) to [bend right=40] (m-5-1);
\draw[-][double] (m-5-1) to [bend right=40] (m-7-1);
\draw[line width=0.4pt] (-1,1.25)--(-0.1,1.25)  (0.5,1.25)--(1.2,1.25);
\draw[line width=0.4pt] (-0.7,-1.3)--(-0.1,-1.3) (0.5,-1.3)--(1.2,-1.3);
\end{tikzpicture}
\end{center}

\begin{center}
\begin{tikzpicture}
\draw node at(0.2,0){$1$} (0.2,0)circle[radius=0.2];
\draw [->,>=angle 90] (0.41,0)--(0.4,0);
\draw (0.42,0.02)--(1.6,0.02) (0.42,-0.02)--(1.6,-0.02);
\draw node at(1.8,0){$2$} (1.8,0)circle[radius=0.2] (2,0)--(3.2,0);
\draw  node at(3.4,0){$2$} (3.4,0)circle[radius=0.2](3.6,0)--(4.8,0);
\draw  node at(5.13,0){$\cdots$} (5.4,0)--(6.6,0);
\draw node{} node at(6.8,0){$2$} (6.8,0)circle[radius=0.2] (7,0)--(8.2,0);
\draw node at(8.4,0){$2$} (8.4,0)circle[radius=0.2];
\draw [->,>=angle 90] (8.61,0)--(8.6,0);
\draw (8.62,0.02)--(9.8,0.02) (8.62,-0.02)--(9.8,-0.02);
\draw node at(10,0){$2$} (10,0)circle[radius=0.2];
\end{tikzpicture}
\end{center}

\ \ \ \ \ \ \ \ \ \ \ \ \ \ \ \ \ \ \

\begin{center}
\begin{tikzpicture}
\matrix (m) [matrix of math nodes, row sep=0.06cm, column sep=0.7cm]
    {     & \hat{\xi}_{1} & \hat{\xi}_{2}  & \cdots & \hat{\xi}_{n-2}  & \hat{\xi}_{n-1} &   \\
        \hat{\xi}_{0}   &   &   &  &    &   & \hat{\xi}_{n} \\
          & \hat{\xi}_{2n-1}  & \hat{\xi}_{2n-2}  & \cdots  & \hat{\xi}_{n+2}  & \hat{\xi}_{n+1} & \\};
\path
(m-2-1) edge (m-1-2) (m-1-2) edge (m-1-3)
(m-1-5) edge (m-1-6) (m-1-6) edge (m-2-7)
(m-2-1) edge (m-3-2) (m-3-2) edge (m-3-3)
(m-3-5) edge (m-3-6) (m-3-6) edge (m-2-7)
(m-1-2) edge[double] (m-3-2) (m-1-3) edge[double] (m-3-3)
(m-1-5) edge[double] (m-3-5) (m-1-6) edge[double] (m-3-6);
\draw (-1.2,0.68)--(-0.2,0.68) (0.45,0.68)--(1.15,0.68);
\draw (-0.95,-0.73)--(-0.2,-0.73) (0.45,-0.73)--(1.15,-0.73);
\end{tikzpicture}
\end{center}

\begin{center}
\begin{tikzpicture}
\draw node at(0.2,0){$4$} (0.2,0)circle[radius=0.2];
\draw (0.4,0.02)--(1.58,0.02) (0.4,-0.02)--(1.58,-0.02);
\draw node at(1.8,0){$4$} (1.8,0)circle[radius=0.2] (2,0)--(3.2,0);
\draw [->,>=angle 90] (1.59,0)--(1.6,0);
\draw  node at(3.4,0){$4$} (3.4,0)circle[radius=0.2](3.6,0)--(4.8,0);
\draw  node at(5.13,0){$\cdots$} (5.4,0)--(6.6,0);
\draw node{} node at(6.8,0){$4$} (6.8,0)circle[radius=0.2] (7,0)--(8.2,0);
\draw node at(8.4,0){$4$} (8.4,0)circle[radius=0.2];
\draw [->,>=angle 90] (8.61,0)--(8.6,0);
\draw (8.62,0.02)--(9.8,0.02) (8.62,-0.02)--(9.8,-0.02);
\draw node at(10,0){$4$} (10,0)circle[radius=0.2];
\end{tikzpicture}
\end{center}

\subsection{Realizations of ${\rm E_6^{(2)}}$ and ${\rm F_4^{(1)}}$ by the pair ($O$,$T$)}

As explained in subsection \ref{s:octa}, when inducing up from the seven irreducible $T$-modules,
there are only five different induced modules: $\hat{\tau}_{0}, \hat{\tau}_{1}, \hat{\tau}_{2}, \hat{\tau}_{1}'$, $\hat{\tau}_{0}'$.
Correspondingly, when restricting down the eight irreducible $O$-modules, there are only
five distinct $T$-modules: $\check{\omega}_0^+,\check{\omega}_1^+,
\check{\omega}_2^+,\check{\omega}_{3}$, $\check{\omega}_{4}$.

The induced modules decompose themselves into sum of irreducible modules as follows:
\begin{align*}
\hat{\tau}_{0}=\omega_0^+\oplus\omega_0^-,  \ \  \ \ \hat{\tau}_{1}=\omega_1^+\oplus\omega_1^-, \ \ \ \ \hat{\tau}_2=\omega_2^+\oplus\omega_2^-,
      \ \ \ \  \hat{\tau}_{1}'=\hat{\tau}_{1}''=\omega_{3}, \ \ \ \ \ \hat{\tau}_{0}'=\hat{\tau}_{0}''=\omega_{4},
\end{align*}
while the restriction also decomposes into irreducible modules:
\begin{align*}
  &\check{\omega}_0^+=\check{\omega}_0^-=\tau_{0},   \ \  \ \ \check{\omega}_1^+=\check{\omega}_1^-=\tau_{1},
  \ \  \ \ \check{\omega}_2^+=\check{\omega}_2^-=\tau_{2},
    \ \  \  \check{\omega}_{3}=\tau_{1}'\oplus\tau_{1}'', \ \ \ \ \check{\omega}_{4}=\tau_{0}'\oplus\tau_{0}''.
\end{align*}
Using the imbedding $\tau_{1}: T\hookrightarrow \mathrm{SU}_2$ and ${\omega}_1^+: O\hookrightarrow \mathrm{SU}_2$, the fusion products are given explicitly as follows,
\begin{align*}
  &\tau_{1}\otimes \check{\omega}_0^+=\check{\omega}_1^+, \ \ \ \ \ \ \ \ \ \ \ \ \ \ \ \ \ \ \ \ \ \ \ \ \
   \omega_1^+\otimes \hat{\tau}_{0}=\hat{\tau}_{1},\\
  &\tau_{1}\otimes \check{\omega}_1^+=\check{\omega}_0^+\oplus\check{\omega}_2^+, \ \ \ \ \ \ \ \ \ \ \ \ \ \ \
   \omega_1^+\otimes \hat{\tau}_{1}=\hat{\tau}_{0}\oplus \hat{\tau}_{2},\\
  &\tau_{1}\otimes \check{\omega}_2^+=\check{\omega}_1^+\oplus \check{\omega}_{3}, \ \ \ \ \ \ \ \  \ \ \ \ \ \ \ \
   \omega_1^+\otimes \hat{\tau}_{2}=\hat{\tau}_{1}\oplus 2\hat{\tau}_{1}',\\
  &\tau_{1}\otimes \check{\omega}_{3}=2\check{\omega}_2^+\oplus \check{\omega}_{4},\ \ \ \ \ \ \ \ \ \ \ \ \ \ \
    \omega_1^+\otimes \hat{\tau}_{1}'=\hat{\tau}_{2}\oplus \hat{\tau}_{0}',\\
  &\tau_{1}\otimes \check{\omega}_{4}=\check{\omega}_{3},\ \ \ \ \ \ \ \ \ \ \ \ \ \ \ \ \ \ \ \ \ \ \ \ \ \ \ \
    \omega_1^+\otimes \hat{\tau}_{0}'=\hat{\tau}_{1}'.
\end{align*}
Therefore the representation graph $\mathcal{R}_{\tau_{1}}(\check{O})$ realizes the Dynkin diagram of type ${\rm E_6^{(2)}}$
and the representation graph $\mathcal{R}_{\omega_1^+}(\hat{T})$ realizes the Dynkin diagram of type ${\rm F_4^{(1)}}$.
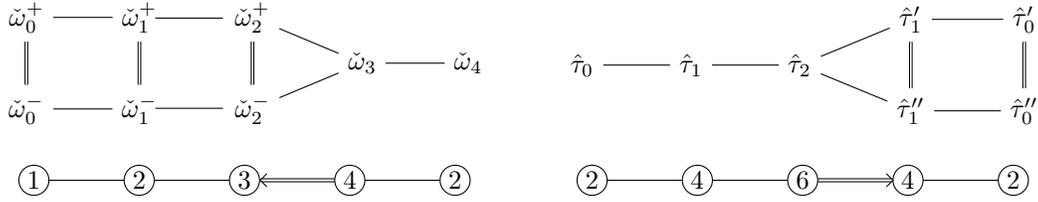
\begin{center}
\begin{tikzpicture}
\matrix (m) [matrix of math nodes, row sep=0.04cm, column sep=0.8cm]
    { \check{\omega}_0^+ & \check{\omega}_1^+ &  \check{\omega}_2^+ &   & \\
        &   &  & \check{\omega}_{3} &  \check{\omega}_{4}   \\
      \check{\omega}_0^- & \check{\omega}_1^- &  \check{\omega}_2^- &   & \\ };
\path
(m-1-1) edge (m-1-2)  (m-1-3) edge (m-2-4)
(m-3-1) edge (m-3-2)  (m-3-3) edge (m-2-4)
(m-2-4) edge (m-2-5)
(m-1-1) edge[double] (m-3-1) (m-1-2) edge[double] (m-3-2)
(m-1-3) edge[double] (m-3-3);
\draw[line width=0.4pt] (-1.2,0.6)--(-0.3,0.6)  (-1.2,-0.6)--(-0.3,-0.6);
\end{tikzpicture}
\ \ \ \ \ \ \
\begin{tikzpicture}
\matrix (m) [matrix of math nodes, row sep=0.05cm, column sep=0.9cm]
    {   &   &   &  \hat{\tau}_{1}' & \hat{\tau}_{0}' \\
      \hat{\tau}_{0}  & \hat{\tau}_{1}  & \hat{\tau}_2 &   &    \\
        &   &    &  \hat{\tau}_{1}'' & \hat{\tau}_{0}'' \\ };
\path
(m-1-4) edge (m-1-5) (m-2-3) edge (m-1-4)
(m-3-4) edge (m-3-5) (m-2-3) edge (m-3-4)
(m-2-1) edge (m-2-2) (m-2-2) edge (m-2-3)
(m-1-4) edge[double] (m-3-4)
(m-1-5) edge[double] (m-3-5);
\end{tikzpicture}
\end{center}

\begin{center}
\begin{tikzpicture}
\draw node at(0.2,0){$1$} (0.2,0)circle[radius=0.2] (0.4,0)--(1.4,0);
\draw node at(1.6,0){$2$} (1.6,0)circle[radius=0.2] (1.8,0)--(2.8,0);
\draw  node at(3,0){$3$} (3,0)circle[radius=0.2];
\draw [->,>=angle 90] (3.21,0)--(3.2,0);
\draw (3.22,0.02)--(4.2,0.02) (3.22,-0.02)--(4.2,-0.02);
\draw  node at(4.4,0){$4$} (4.4,0)circle[radius=0.2] (4.6,0)--(5.6,0)
node at(5.8,0){$2$} (5.8,0)circle[radius=0.2];
\end{tikzpicture}
\ \ \ \ \ \ \ \ \ \ \ \ \ \ \
\begin{tikzpicture}
\draw node at(0.2,0){$2$} (0.2,0)circle[radius=0.2] (0.4,0)--(1.4,0);
\draw node at(1.6,0){$4$} (1.6,0)circle[radius=0.2] (1.8,0)--(2.8,0);
\draw  node at(3,0){$6$} (3,0)circle[radius=0.2];
\draw [->,>=angle 90] (4.2,0)--(4.21,0);
\draw (3.2,0.02)--(4.18,0.02) (3.2,-0.02)--(4.18,-0.02);
\draw  node at(4.4,0){$4$} (4.4,0)circle[radius=0.2] (4.6,0)--(5.6,0)
node at(5.8,0){$2$} (5.8,0)circle[radius=0.2];
\end{tikzpicture}
\end{center}

\subsection{Realizations of ${\rm D_4^{(3)}}$ and ${\rm G_2^{(1)}}$ by the pair ($T$,$D_2$)}

The binary dihedral group $D_2$ has five irreducible modules $\delta_0^{\pm}, \delta_1, \delta_2^{\pm}$. As $D_2$ is a normal subgroup
of the binary tetrahedral group $T$ (with index $3$), the irreducible $D_2$-modules are induced to only three distinct modules of $T$
as follows:
\begin{align*}
&\hat{\delta}_{0}^{+}=\tau_{0}\oplus\tau_{0}'\oplus\tau_{0}'', \ \ \ \ \ \hat{\delta}_1=\tau_{1}\oplus\tau_{1}'\oplus\tau_{1}'', \ \ \ \ \
  \hat{\delta}_{2}^+=\hat{\delta}_{2}^-=\hat{\delta}_{0}^-=\tau_2,
\end{align*}
where $\{\tau_0^{(i)}, \tau_1^{(i)}, \tau_2|i=0, 1, 2\}$ form the complete set of irreducible $T$-modules.

Correspondingly, the restrictions of the 7 irreducible $T$-modules give rise to three distinct $D_2$-modules: 
\begin{align*}
  &\check{\tau}_{0}=\check{\tau}_{0}'=\check{\tau}_{0}''=\delta_{0}^+, \ \ \ \ \ \check{\tau}_{1}=\check{\tau}_{1}'=\check{\tau}_{1}''=\delta_{1},
  \ \ \ \ \ \check{\tau}_2=\delta_{2}^+\oplus\delta_{0}^-\oplus\delta_{2}^-.
\end{align*}

Using the embedding $\delta_1=\check{\tau}_{1}$ and $\tau_1$, we get the fusion rules by those of ${{\rm E_6^{(1)}}}$ and ${\rm D_4^{(1)}}$
respectively as follows:
\begin{align*}
  &\delta_1\otimes \check{\tau}_{0}=\check{\tau}_{1}, \qquad \ \ \ \ \ \ \ \ \
  \tau_{1}\otimes \hat{\delta}_{0}^+ = \hat{\delta}_1, \\
  &\delta_1\otimes \check{\tau}_{1}=\check{\tau}_{0}\oplus \check{\tau}_2, \qquad
  \tau_{1}\otimes \hat{\delta}_1 = \hat{\delta}_{0}^+\oplus 3\hat{\delta}_{2}^+, \\
  &\delta_1\otimes \check{\tau}_2=3\check{\tau}_{1}, \qquad \ \ \ \ \ \ \
  \tau_{1}\otimes \hat{\delta}_{2}^+ = \hat{\delta}_1.
\end{align*}
Thus the representation graphs
$\mathcal{R}_{\delta_1}(\check{T})$
and $\mathcal{R}_{\tau_{1}}(\hat{D}_2)$
realize the Dynkin diagrams  ${\rm D_4^{(3)}}$ and ${\rm G_2^{(1)}}$ respectively.

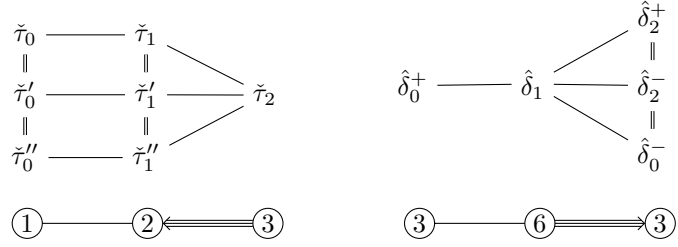
\begin{center}
\begin{tikzpicture}
\matrix (m) [matrix of math nodes, row sep=0.25cm, column sep=1cm]
    { \check{\tau}_{0} & \check{\tau}_{1} &   \\
      \check{\tau}_{0}' & \check{\tau}_{1}' & \check{\tau}_2   \\
      \check{\tau}_{0}'' & \check{\tau}_{1}'' &   \\ };
\path
(m-1-1) edge (m-1-2) (m-1-2) edge (m-2-3)
(m-2-1) edge (m-2-2) (m-2-2) edge (m-2-3)
(m-3-1) edge (m-3-2) (m-3-2) edge (m-2-3)
(m-1-1) edge[double] (m-2-1) (m-1-2) edge[double] (m-2-2)
(m-2-1) edge[double] (m-3-1) (m-2-2) edge[double] (m-3-2);
\end{tikzpicture}
\ \ \ \ \ \ \ \ \ \ \ \
\begin{tikzpicture}
\matrix (m) [matrix of math nodes, row sep=0.25cm, column sep=1cm]
    {  &  & \hat{\delta}_{2}^+  \\
      \hat{\delta}_{0}^+ & \hat{\delta}_1 & \hat{\delta}_{2}^-   \\
       &  &  \hat{\delta}_{0}^- \\ };
\path
(m-1-3) edge (m-2-2)
(m-2-1) edge (m-2-2) (m-2-2) edge (m-2-3)
(m-3-3) edge (m-2-2)
(m-1-3) edge[double] (m-2-3)
(m-3-3) edge[double] (m-2-3);
\end{tikzpicture}
\end{center}

\begin{center}
\begin{tikzpicture}
\draw node at(0.2,0){$1$} (0.2,0)circle[radius=0.2] (0.4,0)--(1.6,0);
\draw node at(1.8,0){$2$} (1.8,0)circle[radius=0.2];
\draw [->,>=angle 90]
(3.2,0.04)--(2.05,0.04) (3.2,-0.04)--(2.05,-0.04) (3.2,0)--(2.0,0);
\draw node{ } node at(3.4,0){$3$} (3.4,0)circle[radius=0.2];
\end{tikzpicture}
\ \ \ \ \ \ \ \ \ \ \ \ \ \ \ \
\begin{tikzpicture}
\draw node at(0.2,0){$3$} (0.2,0)circle[radius=0.2] (0.4,0)--(1.6,0);
\draw node at(1.8,0){$6$} (1.8,0)circle[radius=0.2];
\draw [->,>=angle 90]
(2.0,0.04)--(3.15,0.04) (2.0,-0.04)--(3.15,-0.04) (2.0,0)--(3.2,0);
\draw node{ } node at(3.4,0){$3$} (3.4,0)circle[radius=0.2];
\end{tikzpicture}
\end{center}

\subsection{Realizations of ${\rm A_{2}^{(2)}}$ and ${\rm A_1^{(1)}}$ by the pair ($D_{2}$,$C_{2}$)}

The Dynkin diagram ${\rm A_1^{(1)}}$ was realized by the McKay correspondence, this subsection gives another realization
as a by-product of the McKay-Slodowy correspondence. The binary dihedral group $D_2$ has the cyclic group $C_2$ as a normal subgroup
of index 4, so the induction of the two irreducible $C_2$ modules $\xi_i$ still gives two different modules of $D_2$, while the restriction
of the five irreducible modules of $D_2$ is down to two different modules of $C_2$. The exact relations of the restriction and
induction go as follows:
\begin{align*}
&\check{\delta}_{0}^{\pm}=\check{\delta}_{2}^{\pm}=\xi_{0},  \qquad \qquad \qquad \ \ \ \ \ \ \ \check{\delta}_{1}=2\xi_{1}, \\
& \hat\xi_0=\delta_{0}^+\oplus\delta_{0}^-\oplus\delta_{2}^+\oplus\delta_{2}^-,\qquad  \ \ \ \ \ \hat\xi_1=2\delta_1,
\end{align*}
where as usual the set of irreducible modules are
$D_2^*=\{\delta_0^{\pm}, \delta_1, \delta_2^{\pm}\}$ and $C_2^*=\{\xi_0, \xi_1\}$. 

The embeddings $2\xi_1: C_2\hookrightarrow \mathrm{SU}_2$ and $\delta_1: D_2\hookrightarrow \mathrm{SU}_2$ determine the
following fusion rule (using those of ${\rm D}_4^{(1)}$ and ${\rm A_1^{(1)}}$):
\begin{align*}
   & 2\xi_1\otimes \check{\delta}_{0}^+=\check{\delta}_{1}, \qquad \qquad \qquad \ \ \ \ \ \ \ \ \
      \delta_1\otimes\hat\xi_0=2\hat\xi_1, \\
   & 2\xi_1\otimes \check{\delta}_{1}=4\check{\delta}_{0}^+, \qquad \qquad \qquad \ \ \ \ \ \ \
     \delta_1\otimes\hat\xi_1=2\hat\xi_0.
\end{align*}
Subsequently the representation graphs $\mathcal{R}_{2\xi_1}(\check{D}_{2})$ and $\mathcal{R}_{\delta_1}(\hat{C}_2)$
realize the affine Dynkin diagram ${\rm A_2^{(2)}}$ and the affine Dynkin diagram ${\rm A_{1}^{(1)}}$, respectively.

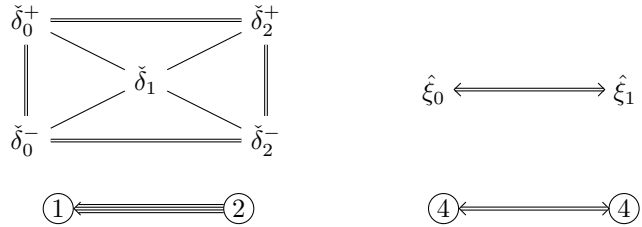
\begin{center}
\ \begin{tikzpicture}
\matrix (m) [matrix of math nodes, row sep=0.18cm, column sep=1cm]
    {  \check{\delta}_{0}^+ &  & \check{\delta}_{2}^+  \\
        & \check{\delta}_1 &   \\
       \check{\delta}_{0}^- &   & \check{\delta}_{2}^-  \\ };
\path
(m-1-1) edge (m-2-2) (m-1-3) edge (m-2-2)
(m-3-1) edge (m-2-2) (m-3-3) edge (m-2-2)
(m-1-1) edge[double] (m-3-1) (m-1-3) edge[double] (m-3-3)
(m-1-1) edge[double] (m-1-3) (m-3-1) edge[double] (m-3-3);
\end{tikzpicture}
\ \ \ \ \ \ \ \ \ \ \ \ \ \ \
\begin{tikzpicture}
\matrix (m) [matrix of math nodes, row sep=0.68cm, column sep=1cm]
  {    &  &   \\
   \hat\xi_0 &   & \hat\xi_1  \\
         &   &   \\ };
\path
(m-2-1) edge[double] (m-2-3);
\draw [->,>=angle 90] (1,0)--(1.01,0);
\draw [->,>=angle 90] (-1,0)--(-1.01,0);
\end{tikzpicture}
\end{center}

\begin{center}
\ \ \ \ \ \begin{tikzpicture}
\draw node{} node at(0.2,0){$1$} (0.2,0)circle[radius=0.2];
\draw node at(2.6,0){$2$} (2.6,0)circle[radius=0.2];
\draw [->,>=angle 90] (0.41,0)--(0.4,0);
\draw
(2.4,0.054)--(0.47,0.054) (2.4,-0.054)--(0.47,-0.054)
(2.4,0.018)--(0.43,0.018) (2.4,-0.018)--(0.43,-0.018);
\end{tikzpicture}
\ \ \ \ \ \ \ \ \ \ \ \ \ \ \ \ \ \ \ \ \ \ \ \
\begin{tikzpicture}
\draw node{} node at(0.2,0){$4$} (0.2,0)circle[radius=0.2];
\draw node at(2.6,0){$4$} (2.6,0)circle[radius=0.2];
\draw [->,>=angle 90] (0.41,0)--(0.4,0);
\draw [->,>=angle 90] (2.4,0)--(2.41,0);
\draw
(2.37,0.02)--(0.43,0.02) (2.37,-0.02)--(0.43,-0.02);
\end{tikzpicture}
\end{center}

\begin{remark}
In subsections $2.2, 2.3, 2.5$ and $2.6$, $\alpha_{\rm \widetilde{A}}=|G:N|^{-1}(\mathrm{dim}\hat\phi_i)_{i\in \hat{\mathrm{I}}}$
and $\alpha_{\rm \widetilde{B}}=(\mathrm{dim}\check\rho_i)_{i\in\check{\mathrm{I}}}$ are unique eigenvectors (up to constants) with zero eigenvalue for
the Cartan matrices ${\rm C_{\widetilde{A}}}$ and ${\rm C_{\widetilde{B}}}$ respectively, while in
subsection $2.4$ and $2.7$, $\alpha_{\rm \widetilde{B}}$ and $\alpha_{\rm \widetilde{A}}$ are eigenvectors with zero eigenvalue for
${\rm C_{\widetilde{A}}}$ and ${\rm C_{\widetilde{B}}}^{T}$ respectively.
Here $\mathrm{dim}\check\rho_i$ $(i\in \check{\mathrm{I}})$
and $|G:N|^{-1}\mathrm{dim}\hat\phi_i$ $(i\in \hat{\mathrm{I}})$ are two sets of
relatively prime integers. 
\end{remark}

\section{Poincar\'{e} series}

\subsection{Poincar\'{e} series for a pair of finite groups $N \unlhd G$}

Let $G$ be a finite group and $V$ be a $G$-module, the Poincar\'e series $m_V(t)$ for the
$G$-invariants in the tensor algebra $T(V)=\displaystyle\bigoplus_{k=0}^{\infty} V^{\otimes k}$
was given by Benkart \cite{Ben} and she
shows that the general Poincar\'e series for a $G$-irreducible
module also admits a nice compact formula similar to the situation of the symmetric tensors. In particular, when $G$ is a finite subgroup
of $\mathrm{SU}_2$, the Poincar\'e series can be computed via the McKay correspondence.
Moreover, the exponents of the affine Lie algebras of simply laced type
are realized via the Poincar\'e series $m_V(t)$, except when $G$ is a cyclic group of odd order,
i.e. the Dynkin diagram ${\rm A}_{2n}^{(1)}$.

To understand the general case and recover the missing
exponents for all affine Lie algebras, we consider the Poincar\'e series for certain modules that are
both $N$-modules and $G$-modules for a given pair of finite groups $N\lhd G$ in view of the McKay-Slodowy correspondence.
We will show that
the Poincar\'{e} series for both $V$ and $V|_N$ in the tensor algebra
$T(V)$ will provide an answer in the general situation. First we recall the following results.

\begin{lemm} {\rm\cite[Cor. 18.7.1]{Kar}} \label{modisom}
Let $\chi_{\phi_i}$ and $\chi_{\phi_j}$
be the characters of $N$ afforded by simple modules $\phi_i$ and $\phi_j$ respectively. Then the induced $G$-modules
$\hat\phi_i \cong \hat\phi_j$ if and only if $\chi_{\phi_i}$ and $\chi_{\phi_j}$ are $G$-conjugate.
\end{lemm}

\begin{lemm} {\rm\cite[Cor. 18.7.5]{Kar}} \label{number}
Let $N$ be a normal subgroup of $G$. 
Then the number of nonisomorphic $G$-modules 
induced from the simple $N$-modules is equal
to the number of conjugacy classes of $G$ contained in $N$.
\end{lemm}

\begin{lemm} {\rm\cite[Cor. 18.11.2]{Kar}}\label{Clifford} (Clifford's theorem)
Let $N$ be a normal subgroup of $G$ and
let $\rho$ be a simple $G$-module. 
There exists a simple submodule $\phi$ of $\check\rho$ and
an integer $e\geq 1$ such that $\check\rho\cong e(\oplus_{t\in T}{}^t\phi)$,
where $H$ is the inertia group of $\phi$,
$T$ is a left transversal for $H$ in $G$ and the conjugates ${}^t\phi\ (t\in T)$
of $\phi$ are pairwise non-isomorphic simple $N$-modules.
\end{lemm}

With these preparations, we have the following result.

\begin{lemm}\label{reinnumber}
 Let $N\unlhd G$, 
and let $\{\rho_i|i\in \mathrm{I_G}\}$ (resp. $\{\phi_i|i\in \mathrm{I_N}\}$) be
the set of pairwise nonisomorphic $\mathbb C$-irreducible modules of $G$ (resp. $N$). 
Let $\{\check\rho_i|i\in \check{\mathrm{I}}\}$ be the set of non-isomorphic $N$-restrictions of irreducible $G$-modules 
with $\check\rho_i \cap \check\rho_j=0$ for $i,j\in\check{\mathrm{I}}$,
let $\{\hat\phi_i|i\in \hat{\mathrm{I}}\}$ be the set of non-isomorphic induced $G$-modules.
Then $|\{\check\rho_i|i\in \check{\mathrm{I}}\}|=|\{\hat\phi_i|i\in \hat{\mathrm{I}}\}|$,
and this common number is equal to $|\Upsilon(N)|$, where $\Upsilon(N)=\Upsilon\cap N$ and $\Upsilon$ is a
fixed set of conjugacy class representatives of $G$.
\end{lemm}

\begin{proof} We define a map $f:  \{\check\rho_i|i\in \check{\mathrm{I}}\}  \longrightarrow  \{\hat\phi_i|i\in \hat{\mathrm{I}}\}$ by
\begin{equation*}
f(\check{\rho}_i)=\hat{\phi}_k \quad \mbox{provided that}\quad \rho_i\lesssim {\hat\phi}_k \ \mbox{and} \ \phi_k\lesssim \check{\rho}_i,
\end{equation*}
where $\lesssim$ means isomorphic to a submodule. By Frobenius reciprocity $(\chi_{\rho_i}, \chi_{\hat{\phi}_k})_G=(\chi_{\check{\rho}_i}, \chi_{{\phi}_k})_H$, so the map is given simply as
$f(\check{\rho}_i)=\hat{\phi}_k$ provided that $\phi_k\lesssim \check{\rho}_i$.

We show the map is well-defined. Suppose $\check\rho_i\cong\check\rho_j$ for some $i,j\in\check{\mathrm{I}}$,
such that $f(\check\rho_i)=\hat\phi_k$ and $f(\check\rho_j)=\hat\phi_l$
with $\phi_k\lesssim \check{\rho}_i$ and $\phi_l\lesssim \check{\rho}_j$ respectively. Then $\phi_k$ and $\phi_l$ are $G$-conjugate by $\mathrm{Lemma}$ \ref{Clifford}.
Namely, $\chi_{\phi_k}$ and $\chi_{\phi_l}$ are $G$-conjugate. By $\mathrm{Lemma}$ \ref{modisom}, $\hat\phi_k\cong\hat\phi_l$.

We will verify $f$ is a bijective. The surjective is guaranteed by the definition of $f$.
Let $\check\rho_i\ncong\check\rho_j$ for any $i,j \in \check{\mathrm{I}}$, such that $f(\check\rho_i)=\hat\phi_k$ and
$f(\check\rho_j)=\hat\phi_l$ with $\phi_k\lesssim \check\rho_i$ and $\phi_l\lesssim \check\rho_j$ respectively.
Since $\check\rho_i\cap\check\rho_j=0$, we have $\phi_k\ncong \phi_l$. Thus $\chi_{\phi_k} \neq \chi_{\phi_l}$.
$N$ is a normal subgroup of $G$ implies that $\chi_{\phi_k}$ and $\chi_{\phi_l}$ are not
$G$-conjugate. By $\mathrm{Lemma}$ \ref{modisom}, there is $\hat\phi_k\ncong\hat\phi_l$. Then $f$ is an injective.

By $f$ is a bijective and $\mathrm{Lemma}$ \ref{number}, we have $|\{\check\rho_i|i\in \check{\mathrm{I}}\}|=|\{\hat\phi_i|i\in \hat{\mathrm{I}}\}|=
|\Upsilon(N)|$. This completes the proof.
\end{proof}

Let $G$ be a finite group equipped with a faithful module.
Steinberg \cite{Ste} studied the decomposition of the tensor product of the
faithful module and any irreducible module in terms of characters.
We have an analogous result for a pair of finite groups $N\unlhd G$ in view of $\mathrm{Lemma}$ \ref{reinnumber}. 

\begin{lemm} \label{eigenvalues}
Let $N\unlhd G$ be a pair of finite groups. Let $\{\check\rho_i|i\in \check{\mathrm{I}}\}$
(resp. $\{\hat\phi_i|i\in \hat{\mathrm{I}}\}$) be the set of $N$-restriction modules of irreducible $G$-modules
(resp. induced $G$-modules of irreducible $N$-modules).
Let $\rho$ be a faithful $N$-module which is also the restriction of a faithful $G$-module with $d=
\chi_{\rho}(1)$. Then,
\begin{enumerate}
  \item The column vectors $(\chi_{\check\rho_i}(g))$ and $(\chi_{\hat\phi_i}(g))$ are respectively the eigenvectors of the matrices $(d\delta_{ji}-a_{ij})$
  and $(d\delta_{ij}-b_{ij})$ with eigenvalue $d-\chi_{\rho}(g)$,
  where $g$ runs through $\Upsilon(N)=N\cap \Upsilon$, and $\Upsilon$ is a set of conjugacy class representatives of $G$.
\item In particular, the column vectors $(\check{d}_i)$ and $(\hat{d}_i)$
  are eigenvectors with eigenvalue $0$ respectively,
  where $\check{d}_i=\chi_{\check\rho_i}(1)$ and $\hat{d}_i=\chi_{\hat\phi_i}(1)$.
\end{enumerate}
\end{lemm}

\begin{theo}\label{thm1}
Let $N\unlhd G$ be a pair of finite groups and
$\{\check\rho_i|i\in \check{\mathrm{I}}\}$ the set of complex $N$-restriction of irreducible $G$-modules.
Let $\mathrm{\widetilde{A}}$ be the adjacency matrix of the representation
graph {\rm $\mathcal{R}_{V}(\check{G})$} and ${\rm M}_1^{i}$ the matrix
$\mathrm{I}-t\mathrm{\widetilde{A}}^{T}$ with the $i$th column replaced by $\underline{\delta}=(1, 0, \ldots, 0)^T\in\mathbb R^{|\mathrm{\check I}|}$.
Assume that {\rm $V$} is a faithful (restriction) {\rm $N$}-module
such that $V\simeq V^*$.
Then the  Poincar\'e series {\rm $\check{m}^{i}(t)=\Sigma_{k\geq 0}\check{m}^
{i}_{k}t^{k}$} of the multiplicities of $\check\rho_i$ in {\rm $T(V)=\oplus_{k \geq 0}V^{\otimes k}$} is given by

{\rm \begin{equation}\label{ps}
   \check{m}^{i}(t)=\frac{\mathrm{det}(\mathrm{M}_1^{i})}{\mathrm{det}(\mathrm{I}-t\mathrm{\widetilde{A}}^{T})}
    =\frac{{\mathrm{det}}(\mathrm{M}_1^{i})}{\prod\limits_{g\in \Upsilon(N)}(1- \chi_{V}(g)t)},
\end{equation}}
where $\chi_{V}$ is the character of $V$ and $\Upsilon(N)=N\cap \Upsilon$, and $\Upsilon$ is a fixed set of conjugacy representatives of $G$.
\end{theo}

\begin{proof} Note that $\check{m}^{i}_{k}$=${\mathrm{dim}}(\mathrm{Hom}_{N}(\check\rho_i,V^{\otimes k}))$ and $\check\rho_i$ is trivial
iff $i=0$,
using the argument of \cite[Thm. 2.1]{Ben} it follows that
\begin{eqnarray}\label{3.1}
  \check{m}^{i}(t) &=& 
   \sum\limits_{k\geq 0}{\mathrm{dim}}(\mathrm{Hom}_{N}
   (\check\rho_i,V^{\otimes k}))t^k \nonumber \\
   &=& \delta_{i,0}+t\sum\limits_{k\geq 1}{\mathrm{dim}}(\mathrm{Hom}_{N}
   (V\otimes \check\rho_i ,V^{\otimes {k-1}}))t^{k-1} \nonumber \\
   &=& \delta_{i,0}+t\sum\limits_{k\geq 1}{\mathrm{dim}}(\sum\limits_j a_{ji}\mathrm{Hom}_{N}
   (\check\rho_j ,V^{\otimes {k-1}}))t^{k-1} \nonumber \\
    &=& \delta_{i,0}+t\sum\limits_j a_{ji}
    \sum\limits_{k\geq 0}{\mathrm{dim}}(\mathrm{Hom}_{N}
   (\check\rho_j,V^{\otimes k}))t^k \nonumber \\
    &=& \delta_{i,0}+t\sum\limits_j a_{ji}
    \check{m}^{j}(t).
\end{eqnarray}

Write $\underline{m}=(\check{m}^i(t))_{i\in \mathrm{\check I}}$, the column vector formed by Poincar\'e series, then \eqref{3.1} can be written as the
matrix identity: $({{\mathrm{I}}-t{\rm\widetilde{A}}^{T}})\underline{m}=\underline{\delta}$.
Then the first equality of \eqref{ps} follows from Cramer's rule.

Assume $\rm{ dimV=d}$, ${\rm \rm{d}}-\chi_{V}(g)$ be all eigenvalues of the matrix ${\rm dI}-{\rm \widetilde{A}}$,
where $g$ runs over the set $\Upsilon(N)$=$\Upsilon\cap N$ and $\Upsilon$ is a fixed set of representatives of conjugacy classes of $G$
(see $\mathrm{Lemma}$ \ref{eigenvalues}$(1)$). Therefore
${\mathrm{det}}(t\mathrm{I}-\mathrm{\widetilde{A}})=
  \Pi_{g\in\Upsilon(N)}(t-\chi_{V}(g))$.
Denote $n=|\check{\mathrm{I}}|=|\Upsilon(N)|$, then
\begin{equation*}
  {\mathrm{det}}(\mathrm{I}-t\mathrm{\widetilde{A}}^{T})=t^{n}{\mathrm{det}}(t^{-1}\mathrm{I}-\mathrm{\widetilde{A}}^{T})=
  t^{n}\prod\limits_{g\in\Upsilon(N)}(t^{-1}-\chi_{V}(g))=
  \prod\limits_{g\in\Upsilon(N)}(1-\chi_{V}(g)t),
\end{equation*}
which is the second identity in \eqref{ps}.
\end{proof}

The following result is obtained similarly as $\mathrm{Theorem}$ \ref{thm1}.

\begin{theo}\label{thm11}
Let $N\unlhd G$ be a pair of finite subgroups and
$\hat\phi_i (i\in\hat{\mathrm{I}})$ the induced
{\rm $G$}-module of an irreducible $N$-module $\phi_i$. 
Let {\rm $V$} be a finite-dimensional self-dual 
{\rm $G$}-module $V\simeq V^*$.
Let $\mathrm{\widetilde{B}}$ be the adjacency matrix of the representation
graph {\rm $\mathcal{R}_{V}(\hat{N})$} and ${\rm M}_2^{i}$ the matrix
$\mathrm{I}-t\mathrm{\widetilde{B}}^{T}$ with the $i$th column replaced by $\underline{\delta}=(1, 0, \ldots, 0)^T\in\mathbb R^{|\mathrm{\hat I}|}$.
Then the  Poincar\'e series {\rm $\hat{m}^{i}(t)=\Sigma_{k\geq 0}\hat{m}^
{i}_{k}t^{k}$} of the multiplicities of $\hat{\phi}_i$ in {\rm $T(V)=\oplus_{k \geq 0}V^{\otimes k}$} is given by

\begin{equation}\label{ps1}
   \hat{m}^{i}(t)=\frac{{\mathrm{det}}(\mathrm{M}_2^{i})}{{\mathrm{det}}(\mathrm{I}-t\mathrm{\widetilde{B}}^{T})}
   =\frac{{\mathrm{det}}(\mathrm{M}_2^{i})}{\prod\limits_{g\in \Upsilon(N)}(1-\chi_{V}(g)t)},
\end{equation}
where $\chi_{V}$ is the character of $V$, $\Upsilon(N)=N\cap \Upsilon$, and $\Upsilon$ is a fixed set of conjugacy representatives of $G$.
\end{theo}


\begin{remark}
If $G=N$, both the identities \eqref{ps} and \eqref{ps1} coincide and specialize to the result \cite[Thm 2.1]{Ben}.
\end{remark}

\subsection{Example $A_4\lhd S_4$}

The alternating group $A_4=\langle(123), (124)\rangle$ $\simeq K_4\rtimes C_3$
and its four irreducible modules $\phi_i$ can be lifted from the
one-dimensional modules of the Klein subgroup $K_4$, where $\dim(\phi_i)=1 (i=0, 1, 2)$ and $\dim(\phi_3)=3$.
Similarly the irreducible $S_4$-modules $\rho_i$ can from induced by those of  $A_4$-modules as follows:
\begin{align*}
  &\hat\phi_0=\rho_{0}^+\oplus\rho_{0}^-, \ \ \ \ \ \ \ \ \ \  \hat\phi_{1}=\hat\phi_{2}=\rho_1,  \ \ \ \ \ \ \ \ \ \
  \hat\phi_3=\rho_{2}^+\oplus\rho_{2}^-.
\end{align*}
$\mathrm{Table}$ $4$ and $5$ list the character tables for $S_4$ and $A_4$.

\ \ $\textbf{Table 4}$ \ \ \ \ \ \ \ \ \ \ \ \ \ \ \ \ \ \ \ \ \  \  \ \ \ \ \ \ \ \ \ \
\ \ \ \ \ \  \  \ \ \ \ \ \ \ \ \ \ \ \ \ \ \ \  \  \ \ \ \ \ \ \ \ \ \
Character table of $S_4$.

\ \ \begin{tabular}{l@{\indent\indent\indent\indent\indent}c@{\indent\indent\indent\indent\indent}
c@{\indent\indent\indent\indent\indent}
c@{\indent\indent\indent\indent\indent}c@{\indent\indent\indent\indent\indent}c@{\indent}}
  \hline
  $\chi \backslash g$ & $1$ & $(12)$ & $(123)$ & $(1234)$ & $(12)(34)$ \\
  $\ \ \ \backslash|C_{G}(g)|$  & $1$ & $6$ & $8$ & $6$ & $3$ \\
  \hline
  $\chi_{\rho_{0}^+}$ & $1$ & $1$ & $1$ & $1$ & $1$ \\
  $\chi_{\rho_{0}^-}$ & $1$ & $-1$ & $1$ & $-1$ & $1$ \\
  $\chi_{\rho_1}$ & $2$ & $0$ & $-1$ & $0$ & $2$ \\
  $\chi_{\rho_{2}^+}$ & $3$ & $1$ & $0$ & $-1$ & $-1$ \\
  $\chi_{\rho_{2}^-}$ & $3$ & $-1$ & $0$ & $1$ & $-1$ \\
  \hline
\end{tabular}

\ \ \ \ \ \ \ \ \ \ \

\ \ \ \ \ \ $\textbf{Table 5}$ \ \ \ \ \ \ \ \ \ \ \ \ \ \ \ \ \ \ \ \ \  \  \ \ \ \ \ \ \ \ \ \
\ \ \ \ \ \  \  \ \ \ \ \ \ \ \ \ \ \ \ \ \ \ \  \  \ \ \ \ \ \ \ \ \ \
Character table of $A_4$.

\ \ \ \ \begin{tabular}{l@{\indent\indent\indent\indent\indent\indent}c@{\indent\indent\indent\indent\indent\indent}
c@{\indent\indent\indent\indent\indent\indent}c@{\indent\indent\indent\indent\indent\indent}c@{\indent}}
  \hline
  $\chi \backslash g$ & $1$ & $(123)$ & $(132)$ & $(12)(34)$  \\
   $\ \ \ \backslash|C_{G}(g)|$ & $1$ & $4$ & $4$ & $3$  \\
  \hline
  $\chi_{\phi_0}$ & $1$ & $1$ & $1$ & $1$  \\
  $\chi_{\phi_{1}}$ & $1$ & $\theta_3$ & $\theta_3^2$ & $1$  \\
  $\chi_{\phi_{2}}$ & $1$ & $\theta_3^2$ & $\theta_3$ & $1$  \\
  $\chi_{\phi_3}$ & $3$ & $0$ & $0$ & $-1$  \\
  \hline
\end{tabular}

~~~~~~~~~~~~~~~~~~~~~~~~~~~~~~~~~~~~~~~~~~~~~~~~~~~~~~~~~

On the other hand, the restriction of $\rho_i$ is similarly given as follows,
\begin{align*}
  &\check{\rho}_{0}^+=\check{\rho}_{0}^-=\phi_{0}, \ \ \ \ \ \ \ \ \ \  \check{\rho}_{1}=\phi_1\oplus\phi_2,  \ \ \ \ \ \ \ \ \ \
  \check{\rho}_{2}^+=\check{\rho}_{2}^-=\phi_{3}.
\end{align*}
Therefore the $A_4$-restrictions of $\rho_i$ form the set $\{\check{\rho}_{0}^+$, $\check{\rho}_1$, $\check{\rho}_{2}^+\}$ while the inductions of $A_4$-irreducible modules
form the set $\{\hat{\phi}_0$, $\hat{\phi}_1$, $\hat{\phi}_3\}$ with $\mathrm{\check{I}}=\{0, 1, 2\}$ and $\mathrm{\hat{I}}=\{0, 1, 3\}$.
Note that  $\phi_3$ is a faithful $A_4$-module and $\check{\rho}_2^{\pm}=\phi_3$, we have the following fusion rules:
\begin{align*}
   & \phi_3\otimes \check{\rho}_{0}^+=\check{\rho}_{2}^+, \qquad \ \ \ \ \ \ \ \ \ \ \ \ \ \ \ \ \ \ \ \
   \rho_{2}^+\otimes \hat{\phi}_0=\hat{\phi}_3, \\
   &\phi_3\otimes \check{\rho}_1=2\check{\rho}_{2}^+, \qquad \ \ \ \ \ \ \ \ \ \ \ \ \ \ \ \ \ \ \
   \rho_{2}^+\otimes \hat{\phi}_1=\hat{\phi}_3, \\
   &\phi_3\otimes \check{\rho}_{2}^+=\check{\rho}_{0}^+\oplus2\check{\rho}_{2}^+\oplus\check{\rho}_1, \qquad
   \rho_{2}^+\otimes \hat{\phi}_3=\hat{\phi}_0\oplus 2\hat{\phi}_3 \oplus 2\hat{\phi}_1.
\end{align*}
The corresponding graphs $\mathcal{R}_{\phi_3}(\check{S}_4)$ and
$\mathcal{R}_{\rho_{2}^+}(\hat{A}_4)$ are shown below, where the numbers inside
the nodes are the degrees of characters.

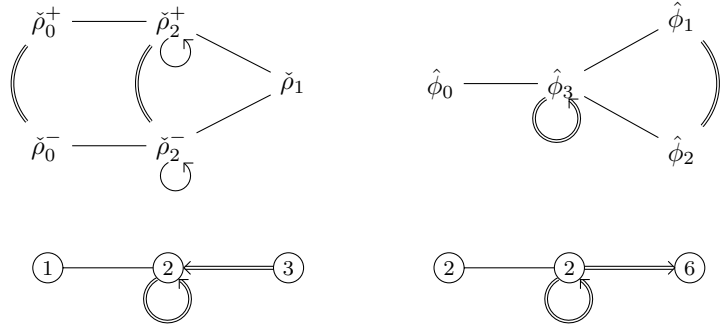
\begin{center}
\begin{tikzpicture}
\matrix (m) [matrix of math nodes, row sep=0.25cm, column sep=1cm]
    { \check{\rho}_{0}^+ & \check{\rho}_{2}^+ &   \\
        &  &  \check{\rho}_{1}  \\
       \check{\rho}_{0}^- & \check{\rho}_{2}^- &  \\ };
\path
(m-1-1) edge (m-1-2) (m-1-2) edge (m-2-3)
(m-3-1) edge (m-3-2) (m-3-2) edge (m-2-3);
\draw  (0,0.6) arc (118:415:2mm);
\draw  (0,-1.05) arc (118:415:2mm);
\draw [->,>=angle 90] (0.21,0.58)--(0.2,0.59);
\draw [->,>=angle 90] (0.21,-1.06)--(0.2,-1.05);
\draw[-][double] (m-1-2) to [bend right=40] (m-3-2);
\draw[-][double] (m-1-1) to [bend right=40] (m-3-1);
\end{tikzpicture}
\ \ \ \ \ \ \ \ \ \ \ \
\begin{tikzpicture}
\matrix (m) [matrix of math nodes, row sep=0.25cm, column sep=1cm]
    {  &  & \hat{\phi}_1  \\
      \hat{\phi}_0 & \hat{\phi}_{3} &    \\
       &  &  \hat{\phi}_2\\ };
\path
(m-2-1) edge (m-2-2)
(m-1-3) edge (m-2-2) (m-3-3) edge (m-2-2);
\draw[-][double] (m-1-3) to [bend left=40] (m-3-3);
\draw [double]  (-0.2,-0.2) arc (118:415:3mm);
\draw [->,>=angle 90] (0.101,-0.2)--(0.1,-0.199);
\draw[color=white]  (0,-1.05) arc (118:415:2mm);
\end{tikzpicture}
\end{center}

\ \ \ \ \ \ \ \ \ \ \ \ \ \ \ \ \ \

\begin{center}
\begin{footnotesize}
 \begin{tikzpicture}
\draw node at(0.2,0){$1$} (0.2,0)circle[radius=0.2] (0.4,0)--(1.6,0);
\draw node at(1.8,0){$2$} (1.8,0)circle[radius=0.2];
\draw [->,>=angle 90] (2.01,0)--(2,0);
\draw(2.02,0.02)--(3.2,0.02) (2.02,-0.02)--(3.2,-0.02);
\draw [->,>=angle 90] (1.941,-0.15)--(1.94,-0.149);
\draw [double]  (1.65,-0.15) arc (118:415:3mm);
\draw node{ } node at(3.4,0){$3$} (3.4,0)circle[radius=0.2];
\end{tikzpicture}
\end{footnotesize}
\ \ \ \ \ \ \ \ \ \ \ \ \ \ \ \
\begin{footnotesize}
\begin{tikzpicture}
\draw node at(0.2,0){$2$} (0.2,0)circle[radius=0.2] (0.4,0)--(1.6,0);
\draw node at(1.8,0){$2$} (1.8,0)circle[radius=0.2];
\draw [->,>=angle 90] (3.2,0)--(3.21,0);
\draw (3.18,0.02)--(2,0.02) (3.18,-0.02)--(2,-0.02);
\draw [->,>=angle 90] (1.941,-0.15)--(1.94,-0.149);
\draw [double]  (1.65,-0.15) arc (118:415:3mm);
\draw node{ } node at(3.4,0){$6$} (3.4,0)circle[radius=0.2];
\end{tikzpicture}
\end{footnotesize}
\end{center}
So the adjacency matrices of the representation graphs $\mathcal{R}_{\phi_3}(\check{S}_4)$
and $\mathcal{R}_{\rho_2^+}(\hat{A}_4)$ are
\begin{align*}
 {\rm \widetilde{A}}=\left(
   \begin{array}{ccc}
     0 & 0 & 1 \\
     0 & 0 & 1 \\
    1 & 2 & 2 \\   \end{array}
  \right)
\ \ \ \ \ \ \  \hbox{and} \ \ \ \ \ \ \
  {\rm \widetilde{B}}=\left(
       \begin{array}{ccc}
         0 & 0 & 1 \\
          0 & 0 & 2 \\
          1 & 1 & 2 \\
        \end{array}
      \right).
    \\
\end{align*}
It follows from $\mathrm{Theorem}$ \ref{thm1} and \ref{thm11} that the
pair $(S_4, A_4)$ gives rise to
${\mathrm{det}}(\mathrm{I}-t\mathrm{\widetilde{A}}^{T})={\mathrm{det}}(\mathrm{I}-t\mathrm{\widetilde{B}}^{T})
=\Pi_{g\in \Upsilon(N)}(1- \chi_{V}(g)t)=(1-3t)(1+t)=1-2t-3t^2$ for $V=\phi_3=\check{\rho}_2^{+}$.
Subsequently the Poincar\'e series of the restrictions of $A_4$-modules
and the induced $S_4$-modules associated with {\rm $T(V)=
\oplus_{k \geq 0}V^{\otimes k}$} are computed by
\begin{eqnarray*}
  & &\check{m}^{0}(t) = \hat{m}^{0}(t) =\frac{1-2t-2t^2}{1-2t-3t^2}=
  1+t^2+2t^3+7t^4+20t^5+61t^6+182t^7+\cdots \nonumber \\
  & &\check{m}^{2}(t) = \hat{m}^{3}(t) =\frac{t}{1-2t-3t^2}=
  t+2t^2+7t^3+20t^4+61t^5+182t^6+547t^7+\cdots \nonumber \\
  & &\check{m}^{1}(t) = 2\hat{m}^{1}(t) =\frac{2t^2}{1-2t-3t^2}=
  2t^2+4t^3+14t^4+40t^5+122t^6+364t^7+\cdots.
\end{eqnarray*}

\section{Poincar\'{e} series for $\mathrm{SU}_2$}

If $N \lhd G$ is any pair of finite subgroups of $\mathrm{SU}_2$ and
$V=\mathbb{C}^2$ in $\mathrm{Section}$ 2, then the Poincar\'{e} series $\check{m}^{0}(t)$ and $\hat{m}^{0}(t)$ are $N$-invariants and $G$-invariants inside the
tensor algebra $T(V)$.

\begin{theo}\label{thm3.1}
Let $N \lhd G$ be a pair of finite subgroups of {\rm$\mathrm{SU}_2$} and $V=\mathbb{C}^2$.
Then the Poincar\'{e} series $\check{m}^{0}(t)$ and $\hat{m}^{0}(t)$
are $N$-invariants $T(V)^{N}$ and $G$-invariants $T(V)^{G}$ in $T(V)=\oplus_{k \geq 0}V^{\otimes k}$ given by
{\rm \begin{align}\label{3.2}
   \check{m}^{0}(t)=\hat{m}^{0}(t)
   =\frac{\mathrm{det}(\mathrm{I}-t\mathrm{A}^{T})}{{\mathrm{det}}(\mathrm{I}-t\mathrm{\widetilde{A}}^{T})}
   =\frac{\mathrm{det}(\mathrm{I}-t\mathrm{B}^{T})}
   {{\mathrm{det}}(\mathrm{I}-t\mathrm{\widetilde{B}}^{T})}
   =\frac{\mathrm{det}(\mathrm{I}-t\mathrm{A}^{T})}
   {\prod\limits_{g\in \Upsilon(N)}(1-\chi_{V}(g)t)}
   =\frac{\mathrm{det}(\mathrm{I}-t\mathrm{B}^{T})}
   {\prod\limits_{g\in \Upsilon(N)}(1-\chi_{V}(g)t)},
\end{align}}
where ${\rm \widetilde{A}}$ (resp. {\rm ${\rm \widetilde{B}}$})
are the adjacency matrices of twisted (resp. untwisted nonsimply laced) affine Dynkin diagrams {\rm $\mathcal{R}_V
(\check{G})$} (reps. {\rm $\mathcal{R}_{V}(\hat{N})$}), and ${\rm A}$ and ${\rm B}$
are the adjacency matrices
of the finite Dynkin diagrams obtained by removing the special node corresponding to the trivial module.
$\chi_{V}$ is the character of $V$ and $\Upsilon(N)=\Upsilon\cap N$, where $\Upsilon$ is a fixed set of conjugacy
class representative of $G$.
\end{theo}

Besides the special vertex, the Poincar\'e series for other vertices of restriction modules and induced modules
of the Dynkin diagrams also have close relationship.


\begin{coro}\label{relation}
Let $(G, N)=(D_{2(n-1)},D_{n-1})$, $(D_n,C_{2n})$, $(O,T)$ and $(T,D_2)$ in $\mathrm{SU}_2$ and $V=\mathbb{C}^2$.
Then the Poincar\'e series $\check{m}^{i}(t)$ and $\hat{m}^{i'}(t)$ for the $N$-restrictions of irreducible modules
and induced modules of $N$-irreducibles in $T(V)=\oplus_{k \geq 0}V^{\otimes k}$ satisfy the following relation:
\begin{align*}
\check{m}^{i}(t)=
\left\{\begin{array}{ll}
  \hat{m}^{i'}(t), & \mbox{$i'$ is a long root in $\mathcal{R}_{V}(\hat{N})$}\\
  |G:N|\hat{m}^{i'}(t),
   & \mbox{$i'$ is a short  root in $\mathcal{R}_{V}(\hat{N})$},
\end{array}\right.
\end{align*}
where $\hat\phi_{i'}=f(\check\rho_i)$, $f$ is the bijective in $\mathrm{Lemma}$ \ref{reinnumber}.
In particular, for $(G,N)=(D_2,C_2)$, $(D_{2n}, C_{2n})$ $(n\geq 2)$ 
and $V=\mathbb{C}^2$, one has that
\begin{align*}
\hat{m}^{i}(t)=
\left\{\begin{array}{ll}
  \check{m}^{0}(t), & \hbox{\emph{i} is the special vertex of $\mathcal{R}_{V}(\check{G})$}   \\
  2\check{m}^{i'}(t), & \hbox{\emph{i}\ (resp. $i'$) in the finite Dynkin diagram of $\mathcal{R}_{V}(\check{G})$ (resp. $\mathcal{R}_{V}(\hat{N})$)}.
\end{array}\right.
\end{align*}
\end{coro}

\begin{remark}

For the pairs of the finite subgroups {\rm $N \lhd G\leq\mathrm{SU}_2$}
in {\rm Section} $2$ and the natural $G$-module $V=\mathbb C^2$, the Poincar\'e series of the multiplicities of the trivial $N$-module inside
$T(V)=\oplus_{k \geq 0}V^{\otimes k}$ 
coincides with the usual Poincar\'e series
of $N$-invariants and $G$-invariants in 
$T(V)=\oplus_{k \geq 0}V^{\otimes k}$.
In this sense $\mathrm{Theorem}$ \ref{thm3.1} is analogous to the
generalized Poincar\'e series $[\widetilde{P}_{(G,N)}(t)]_0$ associated with the symmetric algebras $S(V)$ defined by Stekolshchik \cite{St},
who proved a generalized Ebeling's theorem (cf. \cite{Ebe}) that the generalized Poincar\'{e} series can be written as a quotient of the characteristic polynomials
of finite and affine Coxeter transformations,
\begin{equation*}\label{st}
  [\widetilde{P}_{(G,N)}(t)]_0=\frac{{\rm detM_0}(t)}{{\rm detM}(t)},
\end{equation*}
where ${\rm detM_0}(t)={\rm det(}t^2{\rm I}-\mathbf{C})$, ${\rm detM}(t)={\rm det(}t^2{\rm I}-\mathbf{C}_a)$,
$\mathbf{C}$ and $\mathbf{C}_a$ are the Coxeter transformation and its affine analog respectively. 
\end{remark}

Recall that the Coxeter transformation is the product of all simple reflections of the root system (similar for affine case),
while the spectrum of the Coxeter transformation is closely related with that of the Cartan matrix
\cite{BLM, Col}.
Based on this, Benkart \cite{Ben} showed that the Poincar\'e series for invariants of group $G\leq \mathrm{SU}_2$
in $T(V)$ can be used to get
the exponents and Coxeter number of the simply laced affine Lie algebras (except ${\rm A}_{2n}^{(1)}$).
We now show that all the remaining cases are recovered by the relative Poincar\'e series for
the restriction and induction modules associated with the pairs of subgroups in view of the
McKay-Slodowy correspondence.

The exponents and Coxeter numbers of the twisted and untwisted non-simply laced types are displayed in $\mathrm{Table}$ $6$.

\ \ \ \ \ \ \ \ \ \ \ \ \ \ $\textbf{Table 6}$\ \ \ \ \ \ \ \ \ \ \ \ \ \ \ \ \ \ \ \ \ \ \ \ \ \ \ \ \ \ \ \ \ \
\ \ \ \ \ \ \ \ \ \ \ \ \ \ \ \ \
Exponents and Coxeter number.

\ \ \ \ \ \ \ \ \ \ \ \  \begin{tabular}{l@{\indent\indent\indent\indent}c@{\indent\indent\indent\indent}c@{}}
  \hline
  Dynkin diagrams & Exponents & Coxeter number   \\
  \hline
  ${\rm A}_1$  & $1$ & $2$  \\
  ${\rm B}_n$  & $1,3,5,\ldots,2n-1$ & $2n$  \\
  ${\rm C}_n$ & $1,3,5,\ldots,2n-1$ & $2n$   \\
  ${\rm F}_4$ & $1,5,7,11$ & $12$   \\
  ${\rm G}_2$ & $1,5$ & $6$  \\
  ${\rm A}_{1}^{(1)}$ & $0,1$  & $1$  \\
  ${\rm A}_{2}^{(2)}$ & $0,2$  & $2$  \\
  ${\rm A}_{2\ell}^{(2)}$ & $0,1,\ldots,\ell$  & $\ell$ \\
  ${\rm B}_{2\ell+1}^{(1)}$, ${\rm A}_{4\ell+1}^{(2)}$ & $0,1,\ldots,\ell-1,\ell,\ell,\ell+1,\ldots,2\ell$ & $2\ell$  \\
  ${\rm B}_{2\ell}^{(1)}$, ${\rm A}_{4\ell-1}^{(2)}$ & $0,2,\ldots,2\ell-2,2\ell-1,2\ell,\ldots,2(2\ell-1)$ & $2(2\ell-1)$  \\
  ${\rm C}_\ell^{(1)}$, ${\rm D}_{\ell+1}^{(2)}$ & $0,1,\ldots,\ell$ & $\ell$  \\
  ${\rm F}_4^{(1)}$, ${\rm E}_6^{(2)}$ & $0,2,3,4,6$ & $6$  \\
  ${\rm G}_2^{(1)}$, ${\rm D}_4^{(3)}$ & $0,1,2$ & $2$  \\
  \hline
\end{tabular}

\ \ \ \ \ \ \ \ \ \ \

\begin{lemm}{\rm\cite[Thm. 2]{BLM}}\label{spetrum}
Let $A=(A_{ij})_{l\times l}$ be a generalized Cartan matrix such that $(1)$ $A_{ii}=2,$
$(2)$ $A_{ij}=0 $ if and only if $A_{ji}=0$, 
and $(3)$ the primitive graph of A has no odd cycles.
Let $R$ be the Coxeter transformation of $A$. Then there are $l$ complex numbers $\eta_1,\ldots,\eta_l$
satisfying $\eta_j+ \eta_{l-j+1}=2 \pi \sqrt{-1}$ for $l\leqslant j\leqslant l$,
such that the spectrum of $R$ is $e^{\eta_1},\ldots, e^{\eta_l}$
and the spectrum of $A$ is $4\cosh^2(\eta_1/4),\ldots, 4\cosh^2(\eta_l/4)$.
\end{lemm}

We have the following result to realize the exponents of affine Kac-Moody Lie algebras in relation with the McKay-Slodowy correspondence.
\begin{theo}\label{thm3.6}
Let {\rm $N \lhd G\leq \mathrm{SU}_2$} and {\rm $V=\mathbb{C}^2$}.
Let ${\rm \widetilde{A}}$ (resp. ${\rm {A}}$) be the adjacency matrix of the nonsimply laced affine (resp. finite) Dynkin diagrams
$\mathcal{R}_{V}(\check{G})$ (resp. of $\mathcal{R}_{V}(\check{G})$). 
Let ${\rm \widetilde{\Delta}}$ (resp. ${\rm \Delta}$) be the set of exponents $\widetilde{m}_i$ (resp. $m_i$) of the affine Lie algebra (resp. finite
simple Lie algebra) associated with the Dynkin diagram
and $\widetilde{h}$ (resp. $h$) be the affine (resp. finite) Coxeter number.
Then the Poincar\'{e} series for $N$-invariants and $G$-invariants in
{\rm $T(V)=\oplus_{k \geq 0}V^{\otimes k}$} 
are
{\rm \begin{align}
  \check{m}^{0}(t)=\hat{m}^{0}(t)
  &=\frac{{\mathrm{det}}({\rm I}-t{\rm A}^{T})}{{\mathrm{det}}({\rm I}-t{\rm  \widetilde{A}}^{T})}
   =\frac{{\rm \rm{det}}({\rm I}-t{\rm A}^{T})}
   {\prod\limits_{g\in \Upsilon(N)}(1-\chi_{V}(g)t)}
  =\frac{\prod\limits_{m_i\in \Delta}(1-2\cos(\frac{m_i\pi }{h})t)}
   {\prod\limits_{\widetilde{m}_i\in \widetilde{\Delta}}(1-2\cos(\frac{\widetilde{m}_i \pi}{\widetilde{h}})t)}, \label{3.4}
\end{align}}
where $\chi_{V}$ is the character of $V$,
$\Upsilon(N)=N\cap \Upsilon$, and $\Upsilon$ is a fixed set of conjugacy representatives of $G$.
\end{theo}

\begin{proof} The affine Cartan matrix
$\mathrm{C_{\widetilde{A}}=2I-\widetilde{A}}$ in view of the McKay-Slodowy correspondence (see Sect. 2). 
Note the duality of the exponents $\widetilde{m}_i+\widetilde{m}_{n-i+1}=\widetilde{h} \ (n=|\mathrm{\check{I}}|) $ implies that for each $i\in \mathrm{\check{I}}$,
the index set of $\widetilde{\Delta}$
\begin{align*}
 \frac{2 \widetilde{m}_i \pi \sqrt{-1}}{\widetilde{h}} + \frac{2 \widetilde{m}_{n-i+1} \pi \sqrt{-1}}{\widetilde{h}} = 2\pi \sqrt{-1}.
\end{align*}

It follows from  $\mathrm{Lemma}$ \ref{spetrum} that the roots of the characteristic polynomial of Cartan matrix $\mathrm{C}_{\widetilde{\mathrm{A}}}$ are
\begin{align*}
  4\cosh^2(\frac{\widetilde{m}_1\pi \sqrt{-1}}{2\widetilde{h}}), 4\cosh^2(\frac{\widetilde{m}_2\pi \sqrt{-1}}{2\widetilde{h}}),
\ldots, 4\cosh^2(\frac{\widetilde{m}_n\pi \sqrt{-1}}{2\widetilde{h}}).
\end{align*}
Therefore the roots of the characteristic polynomial of $\mathrm{\widetilde{A}=2I}-\mathrm{C}_{\widetilde{\mathrm{A}}}$ are
\begin{align*}
  4\cosh^2(\frac{\widetilde{m}_1\pi \sqrt{-1}}{2\widetilde{h}})-2, 4\cosh^2(\frac{\widetilde{m}_2\pi \sqrt{-1}}{2\widetilde{h}})-2,
\ldots, 4\cosh^2(\frac{\widetilde{m}_n\pi \sqrt{-1}}{2\widetilde{h}})-2,
\end{align*}
or 
\begin{align*}
  2\cos(\frac{\widetilde{m}_1\pi}{\widetilde{h}}),2\cos(\frac{\widetilde{m}_2\pi}{\widetilde{h}}),
\ldots,2\cos(\frac{\widetilde{m}_n\pi}{\widetilde{h}}).
\end{align*}
Therefore
$${\mathrm{det}}({\rm I}-t{\rm \widetilde{A}}^{T})={\mathrm{det}}({\rm I}-t{\rm \widetilde{A}})=
\prod\limits_{\widetilde{m}_i\in \widetilde{\Delta}}(1-2\cos(\frac{\widetilde{m}_i \pi}{\widetilde{h}})t).$$

Similarly the adjacency matrix $A$ of the finite Dynkin diagram {\rm $\mathcal{R}_{V}(G)$} also has the
property
$${\mathrm{det}}({\rm I}-t{\rm  A}^{T})=\prod\limits_{m_i\in \Delta}(1-2\cos(\frac{m_i \pi}{h})t).$$
\end{proof}

\begin{remark}
As a consequence of formula \eqref{3.4}, we see that the character values
$\chi_{V}(g)$ ($g\in \Upsilon(N)$) coincide with $2\cos(\widetilde{m}_i \pi/\widetilde{h})$ ($i\in {\rm \check{I}}$).
In particular, when $N=1$, the formula \eqref{3.4} give rise to the exponents of all affine Lie algebras
except ${\rm A}_{2n}^{(1)}$.
\end{remark}

\subsection{Closed form of Poincar\'e series for $N$-invariants and $G$-invariants}

In this section, we give closed-form expressions of
the Poincar\'e series $\check{m}^{0}(t)=\hat{m}^{0}(t)$
for {\rm $N \lhd G\leq \mathrm{SU}_2$} and $V\cong \mathbb{C}^2$ in $T(V)=\oplus_{k \geq 0}V^{\otimes k}$. The
closed-form expressions of $\check{m}^{0}(t)=\hat{m}^{0}(t)$ as the trivial $N$-module in $T(V)$
have been considered in \cite{Ben},  we revisit the closed-form expressions from the viewpoint of the adjacency matrices
${\rm A}$ and ${\rm \widetilde{A}}$ of 
non-simply laced Dynkin diagrams.
For the three pairs $(O,T)$,
$(T,D_2$), $(D_2,C_2$),
the Poincar\'e series $\check{m}^{i}(t)$ and
$\hat{m}^{i}(t)$ are given for $i$ running through $\mathrm{\check{I}}$ and $\mathrm{\hat{I}}$, respectively.

It is known that the Cartan matrices of finite Dynkin diagrams are closely related with the Chebyshev polynomials of the first and second kinds
\cite{Ben,D,Kor}. We also start by recalling  some simple facts about Chebyshev polynomials (also see \cite{Ri}).

The Chebyshev polynomials of the first kind ${\rm T}_n(t)$
and the second kind ${\rm U}_n(t)$ are recursively defined by:
\begin{align}
  &{\rm T}_0(t)=1, \ \ \ \ {\rm T}_1(t)=t,  \ \ \ \ \ {\rm T}_{n+1}(t)=2t{\rm T}_{n}(t)-{\rm T}_{n-1}(t)
  \ \   {\rm for \ all} \ n\geq1, \label{20} \\
  &{\rm U}_0(t)=1, \ \ \ \ {\rm U}_1(t)=2t,  \ \ \ \ {\rm U}_{n+1}(t)=2t{\rm U}_{n}(t)-
  {\rm U}_{n-1}(t) \ \   {\rm for \ all} \  n\geq1. \label{21}
\end{align}
The additive closed forms of the polynomials ${\rm T}_n(t)$ and ${\rm U}_n(t)$ are
\begin{align}
  {\rm T}_{n}(t)&= {\sum\limits_{i=0}^{\lfloor n/2\rfloor}}{\binom{n}{2i}}t^{n-2i}(t^2-1)^i
  =t^n{\sum\limits_{i=0}^{\lfloor n/2\rfloor}}{\binom{n}{2i}}(1-t^{-2})^i, \label{22} \\
  {\rm U}_{n}(t)&= {\sum\limits_{i=0}^{\lfloor n/2\rfloor}}(-1)^i{\binom{n-i}{i}}(2t)^{n-2i}. \label{24}
\end{align}
In addition, the multiplicative expressions of ${\rm T}_n(t)$ and ${\rm U}_n(t)$ are the following well known formulas:
\begin{align}
{\rm T}_{n}(t)&= 2^{n-1}\prod\limits_{i=1}^n\left(t-{\cos}\left(\frac{(2i-1)\pi}{2n}\right)\right), \label{23} \\
{\rm U}_{n}(t)&=2^{n}\prod\limits_{i=1}^n\left(t-{\rm \cos}\left(\frac{\pi i}{n+1}\right)\right).  \label{25}
\end{align}
Moreover, ${\rm T}_n(t)$ and ${\rm U}_n(t)$ are related by
\begin{eqnarray}\label{3.14}
 {\rm T}_{n}(t)&=&{\rm U}_{n}(t)-t{\rm U}_{n-1}(t).
\end{eqnarray}

\subsubsection{The pair ($D_{2(n-1)}$,$D_{n-1}$)}

From Sect. \ref{e:AB}, the imbedding of the subgroup $D_{n-1}$ is  $\pi=\delta_1'\cong\mathbb{C}^2$,
so the set $\Upsilon(D_{n-1})=\{\pm1, (x^2)^i\ (i=1,\ldots, n-2),y\}$.
Since $\chi_{\delta_1'}((x^2)^i)=\theta_{4(n-1)}^{2i}+\theta_{4(n-1)}^{-2i}=2\cos(\pi i/(n-1))$ and
$\chi_{\delta_1'}(y)=0$, thus the denominator in formula \eqref{3.2} is
\begin{align}\label{d3}
 \mathrm{det}(\mathrm{I}-t\mathrm{\widetilde{A}}^{T})
   =\prod\limits_{g\in \Upsilon(D_{n-1})}(1-\chi_{V}(g)t)
  =(1-4t^2)\prod\limits_{i=1}^{n-2}\left(1-2\cos\left(\frac{\pi i}{n-1}\right)t\right).
\end{align}

The pair ($D_{2(n-1)}$,$D_{n-1}$) realizes the
the twisted affine Dynkin diagram of type ${\rm A_{2n-1}^{(2)}}$.
Removing the special vertex 
the finite Dynkin diagram of type ${\rm C}_{n}$ is obtained. Let ${\rm A}$ be the adjacency matrix of type ${\rm C}_{n}$, then
\begin{align*}
{\rm A}^{T}=\left(
      \begin{array}{ccccc}
        0 & 1 & 0 & \cdots & 0 \\
        1 & 0 & 1 & \cdots & 0 \\
        \vdots & \ddots & \ddots & \ddots & 0 \\
        0 & \cdots & 1 & 0 & 2 \\
        0 & 0 & \cdots & 1 & 0 \\
      \end{array}
    \right).
\end{align*}
Set ${\rm c}_{n-1}(t)=\det(\mathrm{I}-t\mathrm{A}^{T})$, then
${\rm c}_0(t)=1$, ${\rm c}_1(t)=1-2t^2$. Using cofactor expansion it is easy to get the 
recursive relation
\begin{equation}\label{b11}
  {\rm c}_{n+1}(t)={\rm c}_n(t)-t^2{\rm c}_{n-1}(t) \ \ {\rm for} \ \ n\geq1.
\end{equation}
Subsequently one has that 
\begin{align}\label{b1}
  {\rm c}_{n-1}(t)=2t^{n}{\rm T}_{n}\left(\frac{t^{-1}}{2}\right),
\end{align}
where ${\rm T}_{n}(t)$ is the Chebyshev polynomial of the first kind. Therefore we have that
\begin{align}
  {\rm c}_{n-1}(t) &= 2t^{n}\left(\frac{t^{-1}}{2}\right)^{n}{\sum\limits_{i=0}^{\lfloor {n}/2\rfloor}}
  {\binom{n}{2i}}\left(1-\left(\frac{t^{-1}}{2}\right)^{-2}\right)^i \nonumber\\
   &= 2^{(1-n)}{\sum\limits_{i=0}^{\lfloor {n}/2\rfloor}}{\binom{n}{2i}}(1-4t^{2})^i,  \label{3.29}
\end{align}
meanwhile the  multiplicative expression of $\mathrm{T}_n(t)$ gives that 
\begin{align}
  {\rm c}_{n-1}(t) &= 2t^{n}\left(2^{n-1}\prod\limits_{i=1}^n\left(\frac{t^{-1}}{2}-{\cos}\left(\frac{(2i-1)\pi}{2n}\right)\right)\right) \nonumber\\
  &= \prod\limits_{i=1}^{n}\left(1-2{\cos}\left(\frac{\left(2i-1\right)\pi}{2n}\right)t\right). \label{3.30}
\end{align}

The formula \eqref{d3} is also related with the Chebyshev polynomial. In fact, using the Laplace expansion and \eqref{b1} we have that
\begin{align*}
  {\rm det}({\rm I}-t{\rm \widetilde{A}}^{T}) &= {\rm c}_{n-1}(t)-t^2{\rm c}_{n-3}(t) \\
   & = 2t^{n}\left({\rm T}_{n}\left(\frac{t^{-1}}{2}\right)-{\rm T}_{n-2}\left(\frac{t^{-1}}{2}\right)\right).
\end{align*}

It follows from \eqref{3.14}, \eqref{21} and \eqref{24} that
\begin{align*}
  {\rm T}_{n}(t)-{\rm T}_{n-2}(t) 
  &=t{\rm U}_{n-1}(t)+t{\rm U}_{n-3}(t)-2{\rm U}_{n-2}(t) \\
  &=(2t^2-2){\rm U}_{n-2}(t) \\
  &=(2t^2-2){\sum\limits_{i=0}^{\lfloor (n-2)/2\rfloor}}(-1)^i{\binom{n-2-i}{i}}(2t)^{n-2-2i}.
\end{align*}
Therefore
\begin{align*}
{\rm det}({\rm I}-t{\rm \widetilde{A}}^{T})
   &=(1-4t^2){\sum\limits_{i=0}^{\lfloor (n-2)/2\rfloor}}(-1)^i{\binom{n-2-i}{i}}t^{2i}.
\end{align*}

Summarizing these, we have proved the following result.
\begin{theo}\label{thm3.23} Associated with the pair $\rm D_{n-1}\lhd D_{2(n-1)}\leq \mathrm{SU}_2$,
the Poincar\'e series
for {\rm $D_{n-1}$}-invariants and {\rm $D_{2(n-1)}$}-invariants
in {\rm $T(V)$} 
are given by
{\rm \begin{equation}\label{3.24}
  \check{m}^{0}(t)=\hat{m}^{0}(t)=
  \frac{\prod\limits_{i=1}^{n}\left(1-2{\cos}\left(\frac{\left(2i-1\right)\pi}{2n}\right)t\right)}
  {(1-4t^2)\prod\limits_{i=1}^{n-2}\left(1-2\cos\left(\frac{\pi i}{n-1}\right)t\right)}
  =\frac{2^{(1-n)}{\sum\limits_{i=0}^{\lfloor {n}/2\rfloor}}{\binom{n}{2i}}(1-4t^{2})^i}
  {(1-4t^2){\sum\limits_{i=0}^{\lfloor (n-2)/2\rfloor}}(-1)^i{\binom{n-2-i}{i}}t^{2i}}.
\end{equation}}
\end{theo}


\subsubsection{The pair ($D_{n}$,$C_{2n}$)}

Recall Sect. \ref{s:typeDC},  the imbedding of $D_n$ is $\pi=\xi_1+\xi_{-1}\cong\mathbb{C}^2$, so
$\Upsilon(C_{2n})=\{\pm1, x^i(i=1,\ldots, n-1)\}$.
Then $\chi_{\xi_1+\xi_{-1}}(\pm 1)=\pm2$ and $\chi_{\xi_1+\xi_{-1}}(x^i)=2\cos(\pi i/n)$
imply that 
\begin{align}\label{c1}
 \mathrm{det}(\mathrm{I}-t\mathrm{\widetilde{A}}^{T})
  ={\prod_{g\in \Upsilon(C_{2n})}(1-\chi_{V}(g)t)}
  =(1-4t^2)\prod_{i=1}^{n-1}\left(1-2\cos\left(\frac{\pi i}{n}\right)t\right).
\end{align}

As the pair ($D_{n}$,$C_{2n}$) realizes the twisted affine Dynkin diagram ${\rm D}_{n+1}^{(2)}$.
The adjacency matrix ${\rm A}$ of the finite Dynkin diagram is 
${\rm B}_{n}$ after removing the special vertex of ${\rm D}_{n+1}^{(2)}$.
As ${\rm B}_{n}$ is the dual of ${\rm C}_{n}$, the denominator
${\rm b}_{n-1}(t):={\mathrm{det}}({\rm I}-t{\rm{A}}^{T})=c_{n-1}(t)$ given by the same formulas
\eqref{b1}, \eqref{3.29} and \eqref{3.30}.

Using the cofactor expansion and combining with \eqref{b1}, \eqref{3.14} and \eqref{21} it follows that
\begin{align*}
  {\rm det}({\rm I}-t{\rm \widetilde{A}}^{T}) &= {\rm b}_{n-1}(t)-2t^2{\rm b}_{n-2}(t) \nonumber \\
   & = 2t^{n}{\rm T}_{n}\left(\frac{t^{-1}}{2}\right)-2t^2\left(2t^{n-1}{\rm T}_{n-1}\left(\frac{t^{-1}}{2}\right)\right) \nonumber  \\
   & = (t^{n-1}-4t^{n+1}){\rm U}_{n-1}\left(\frac{t^{-1}}{2}\right)\\
   &=(1-4t^2){\sum\limits_{i=0}^{\lfloor (n-1)/2\rfloor}}(-1)^i{\binom{n-1-i}{i}}t^{2i},
\end{align*}
where we have used identity \eqref{24}. This then proves the following result.


\begin{theo}\label{thm3.44} For the pair ${\rm C_{2n}\lhd D_{n}}\leq \mathrm{SU}_2$ and $V=\mathbb C^2$, the
Poincar\'e series for
{\rm $C_{2n}$}-invariants and {\rm $D_{n}$}-invariants
in {\rm $T(V)=\oplus_{k \geq 0}V^{\otimes k}$} can be given by
{\rm \begin{equation}\label{3.45}
  \check{m}^{0}(t)=\hat{m}^{0}(t)=
  \frac{\prod\limits_{i=1}^{n}\left(1-2{\cos}\left(\frac{\left(2i-1\right)\pi}{2n}\right)t\right)}
  {(1-4t^2)\prod\limits_{i=1}^{n-1}\left(1-2\cos\left(\frac{\pi i}{n}\right)t\right)}
  =\frac{2^{(1-n)}{\sum\limits_{i=0}^{\lfloor {n}/2\rfloor}}{\binom{n}{2i}}(1-4t^{2})^i}
  {(1-4t^2){\sum\limits_{i=0}^{\lfloor (n-1)/2\rfloor}}(-1)^i{\binom{n-1-i}{i}}t^{2i}}.
\end{equation}}
\end{theo}

\begin{remark}
The two-dimensional natural module of $C_{2n}$ gives rise to the same Poincar\'e series of invariants
for the two pairs of subgroups ($D_{n}$,$C_{2n}$) and $(D_{2n}$,$C_{2n})$, so the equality \eqref{3.45}
gives the Poincar\'e series of invariants for $(D_{2n}$,$C_{2n})$. We remark that
\eqref{3.45} is also given as a reduced form
of the Poincar\'e series for {\rm $C_{2n}$}-invariant
in {\rm $T(V)$} \cite[Thm. 3.23]{Ben}.

\end{remark}

\subsubsection{The pairs  $(O,T)$, $(T,D_2$) and ($D_{2}$,$C_{2}$) }

\qquad

The Poincar\'e series $\check{m}^0(t)=\hat{m}^{0}(t)$ for  these pairs of subgroups
$N \lhd G\leq\mathrm{SU}_2$ and $V\cong \mathbb{C}^2$ can be computed using $\mathrm{Theorems}$ \ref{thm3.1}
or \ref{thm3.6}.


We list  the Poincar\'e series $\check{m}^{0}(t)=\hat{m}^{0}(t)$ for the
pairs $(O,T)$, $(T,D_2$), ($D_{2}$,$C_{2}$) in order as follows.
\begin{eqnarray*}
  \check{m}^{0}(t)=\hat{m}^{0}(t)&=&\frac{{\mathrm{det}}({\rm I}-t{\rm A}^{T})}{{\prod_{g\in \Upsilon(T)}(1-{\rm \chi_{V}}(g)t)}}
   =\frac{\left(1-4\cos^2\left(\frac{\pi}{12}\right)t^2\right)\left(1-4\cos^2\left(\frac{5\pi}{12}\right)t^2\right)}
   {(1-2t)(1+2t)(1-t)(1+t)}  \\
   &=& \frac{1-4t^2+t^4}{1-5t^2+4t^4}=1+t^2+2t^4+6t^6+22t^8+86t^{10}+\cdots \\
  \check{m}^{0}(t)=\hat{m}^{0}(t)&=&\frac{{\mathrm{det}}({\rm I}-t{\rm A}^{T})}{{\prod_{g\in \Upsilon(D_2)}(1-{\rm \chi_{V}}(g)t)}}
   =\frac{\left(1-4\cos^2\left(\frac{\pi}{6}\right)t^2\right)}
   {(1-2t)(1+2t)}  \\
   &=& \frac{1-3t^2}{1-4t^2}=1+t^2+4t^4+16t^6+64t^8+256t^{10}+\cdots \\
  \check{m}^{0}(t)=\hat{m}^{0}(t)&=&\frac{{\mathrm{det}}({\rm I}-t{\rm A}^{T})}{{\prod_{g\in \Upsilon(C_2)}(1-{\rm \chi_{V}}(g)t)}}
   =\frac{\left(1-2\cos\left(\frac{\pi}{2}\right)t\right)}
   {(1-2t)(1+2t)}  \\
   &=& \frac{1}{1-4t^2}=1+4t^2+16t^4+64t^6+256t^8+1024t^{10}+\cdots. \\
\end{eqnarray*}

If $i$ is not the special vertex of the Dynkin diagrams realized by  $(O,T)$, $(T,D_2$), ($D_{2}$,$C_{2}$),
the Poincar\'e series $\check{m}^{i}(t)$ and
$\hat{m}^{i}(t)$ can be directly worked out by applying
$\mathrm{Theorem }$ \ref{thm1} and \ref{thm11}.
They also can be computed by combining the
results of the series $\check{m}^{0}(t)=\hat{m}^{0}(t)$
and the identity $\check{m}^{i}(t)=\delta_{i,0}+t\sum\limits_{j} a_{ji}\check{m}^{j}(t)$
(resp. $\hat{m}^{i}(t)=\delta_{i,0}+t\sum\limits_{j}b_{ji}\hat{m}^{j}(t)$).
The Poincar\'e series of the
restriction of the $T$-modules $\check\omega_i$ and induced $O$-modules $\hat\tau_i$
in {\rm $T(V)=\oplus_{k \geq 0}V^{\otimes k}$} are
\begin{eqnarray*}
  & & \check{m}^{1}(t) = \hat{m}^{1}(t)=\frac{1}{t}(\check{m}^{0}(t)-1)=
  t+2t^3+6t^5+22t^7+86t^9+\cdots  \\
  & & \check{m}^{2}(t) = \hat{m}^{2}(t) =\frac{1}{t}\check{m}^{1}(t)-\check{m}^{0}(t)=
  t^2+4t^4+16t^6+64t^8+256t^{10}+\cdots \\
  & & \check{m}^{3}(t) = 2\hat{m}^{1'}(t) =\frac{1}{t}\check{m}^{2}(t)-\check{m}^{1}(t)=
  2t^3+10t^5+42t^7+170t^9+\cdots \\
  & & \check{m}^{4}(t) = 2\hat{m}^{0'}(t)=t\check{m}^{3}(t)=
  2t^4+10t^6+42t^8+170t^{10}+\cdots.
\end{eqnarray*}
The Poincar\'e series of the restriction of $D_2$-modules $\check\tau_i$ and induced $T$-modules $\hat\delta_i$
in {\rm $T(V)=\oplus_{k \geq 0}V^{\otimes k}$} are
\begin{eqnarray*}
  & & \check{m}^{1}(t)=\hat{m}^{1}(t)=\frac{1}{t}(\check{m}^{0}(t)-1)=
  t+4t^3+16t^5+64t^7+256t^9+\cdots   \\
  & & \check{m}^{2}(t)=3\hat{m}^{2}(t)=3t\check{m}^{1}(t)=
  3t^2+12t^4+48t^6+192t^8+\cdots.
\end{eqnarray*}
The Poincar\'e series of the
restricted $C_2$-module $\check\delta_1$ and the induced $D_2$-module $\hat\xi_1$ in {\rm $T(V)=\oplus_{k \geq 0}V^{\otimes k}$} is
\begin{eqnarray*}
  2\check{m}^{1}(t)= \hat{m}^{1}(t)=\frac{1}{t}(\hat{m}^{0}(t)-1)=4t+16t^3+64t^5+256t^7+1024t^9+\cdots.
\end{eqnarray*}

We remark that Benkart also studied the non-simply laced untwisted types (such as ${\rm B}_n^{(1)}$ and ${\rm C}_n^{(1)}$) in \cite{Ben2}.

\vskip30pt \centerline{\bf ACKNOWLEDGMENT}

N. Jing would like to thank the partial support of
Simons Foundation grant 523868 and NSFC grant 11531004. 
H. Zhang would like to thank the support of NSFC grant 11871325 and 11726016.
\bigskip

\bibliographystyle{amsalpha}

\end{document}